\documentclass[a4paper]{article}
\usepackage{makecell}

\usepackage[numbers,sort&compress]{natbib}

\makeatletter
\let\c@figure\c@table
\let\ftype@figure\ftype@table

\usepackage{enumerate}
\usepackage{calc}

\usepackage{lmodern} 

\usepackage{graphicx} 

\usepackage{amsmath, accents}
\usepackage{amssymb,tikz}
\usepackage{pict2e}
\usepackage{amsthm}
\usepackage{amscd}
\usepackage{mathrsfs}
\usepackage{todonotes}
\usetikzlibrary{calc} 

\usepackage{yfonts}

\usepackage{marvosym}
\usepackage[percent]{overpic}

\newtheorem{theorem}{Theorem} [section]
\newtheorem{proposition}[theorem]{Proposition}	

\newtheorem{lemma}[theorem]{Lemma}		
\newtheorem{example}{Example}[section]

\theoremstyle{definition}

\newtheorem{remark}{Remark}

\DeclareMathOperator{\ord}{ord}

\newcommand{\C}{\mathbb{C}}

\newcommand{\R}{\mathbb{R}}
\newcommand{\Z}{\mathbb{Z}}
\newcommand{\N}{\mathbb{N}}

\newcommand{\re}{\text{\upshape Re\,}}
\newcommand{\im}{\text{\upshape Im\,}}
\newcommand{\res}{\text{\upshape Res\,}}

\let\oldbibliography\thebibliography
\renewcommand{\thebibliography}[1]{\oldbibliography{#1}
\setlength{\itemsep}{-0.5pt}}

\def\Xint#1{\mathchoice
{\XXint\displaystyle\textstyle{#1}}%
{\XXint\textstyle\scriptstyle{#1}}%
{\XXint\scriptstyle\scriptscriptstyle{#1}}%
{\XXint\scriptscriptstyle\scriptscriptstyle{#1}}%
\!\int}
\def\XXint#1#2#3{{\setbox0=\hbox{$#1{#2#3}{\int}$}
\vcenter{\hbox{$#2#3$}}\kern-.5\wd0}}

\def\dashint{\;\Xint-}

\usepackage{tikz}
\usetikzlibrary{arrows}
\usetikzlibrary{decorations.pathmorphing}
\usetikzlibrary{decorations.markings}
\usetikzlibrary{patterns}
\usetikzlibrary{automata}
\usetikzlibrary{positioning}
\usepackage{tikz-cd}
\tikzset{->-/.style={decoration={
				markings,
				mark=at position #1 with {\arrow{latex}}},postaction={decorate}}}
	
	\tikzset{-<-/.style={decoration={
				markings,
				mark=at position #1 with {\arrowreversed{latex}}},postaction={decorate}}}
				
\tikzset{
	master/.style={
		execute at end picture={
			\coordinate (lower right) at (current bounding box.south east);
			\coordinate (upper left) at (current bounding box.north west);
		}
	},
	slave/.style={
		execute at end picture={
			\pgfresetboundingbox
			\path (upper left) rectangle (lower right);
		}
	}
}

\usetikzlibrary{shapes.misc}\tikzset{cross/.style={cross out, draw, 
         minimum size=2*(#1-\pgflinewidth), 
         inner sep=0pt, outer sep=0pt}}

\usepackage{pgfplots}
	
\allowdisplaybreaks	

\numberwithin{equation}{section}

\def\bigO{{\cal O}}

\makeatletter
\newcommand{\oset}[3][0ex]{%
  \mathrel{\mathop{#3}\limits^{
    \vbox to#1{\kern-2\ex@
    \hbox{$\scriptstyle#2$}\vss}}}}
\makeatother

\usepackage{tocloft}
\usepackage[colorlinks=true]{hyperref}
\hypersetup{urlcolor=blue, citecolor=red, linkcolor=blue}

\begin{document}
\title{Balayage of measures: behavior near a cusp}
\author{Christophe Charlier\footnote{Centre for Mathematical Sciences, Lund University, 22100 Lund, Sweden. E-mail: christophe.charlier@math.lu.se
} \, and Jonatan Lenells\footnote{Department of Mathematics, KTH Royal Institute of Technology, 10044 Stockholm, Sweden. E-mail: jlenells@kth.se}}

\maketitle

\begin{abstract}
Let $\mu$ be a positive measure supported on a domain $\Omega$. We consider the behavior of the balayage measure $\nu:=\mathrm{Bal}(\mu,\partial \Omega)$ near a point $z_{0}\in \partial \Omega$ at which $\Omega$ has an outward-pointing cusp. Assuming that the order and coefficient of tangency of the cusp are $d>0$ and $a>0$, respectively, and that $d\mu(z) \asymp |z-z_{0}|^{2b-2}d^{2}z$ as $z\to z_0$ for some $b > 0$, we obtain the leading order term of $\nu$ near $z_{0}$. This leading term is universal in the sense that it only depends on $d$, $a$, and $b$. We also treat the case when the domain has multiple corners and cusps at the same point.
Finally, we obtain an explicit expression for the balayage of the uniform measure on the tacnodal region between two osculating circles, and we give an application of this result to two-dimensional Coulomb gases. 
\end{abstract}




%
%


\section{Introduction}
The balayage $\nu := \mathrm{Bal}(\mu,\partial \Omega)$ of a measure $\mu$ defined on a domain $\Omega$ is a measure supported on $\partial \Omega$ whose potential outside $\Omega$ coincides with that of $\mu$ (up to a constant if $\Omega$ is unbounded). The balayage measure is conveniently defined in terms of harmonic measure: If $\Omega$ is an open subset of the extended complex plane $\C^* = \C \cup \{\infty\}$ and $\mu$ is a non-negative measure on $\Omega$, then $\nu$ is the measure on $\partial \Omega$ defined whenever it exists by
\begin{align}\label{nuharmonicmeasure}
\nu(E) = \int_\Omega \omega(z, E, \Omega) d\mu(z) \qquad \text{for Borel subsets $E$ of $\partial \Omega$},
\end{align}
where $\omega(z, E, \Omega)$ is the harmonic measure of $E$ at $z$ in $\Omega$. 
Balayage measures are important in potential theory, where they can be used, for example, to solve the Dirichlet problem, to calculate capacities, to study obstacle problems, to characterize regular boundary points and polar sets, to analyze the boundary behavior of superharmonic functions, to compute hole probabilities for two-dimensional point processes, and to analyze the zeros of orthogonal polynomials, see e.g. \cite{AR2017, C2023, D2001, Land1972, G2004, BBMP2009, Z2023b, GR2018, GK2021, NW2023}. Balayage theory has been generalized in various ways: a theory of balayage on locally compact spaces is developed in \cite{Z2022, Z2023, Z2023b}; discrete balayage measures were applied to the boundary sandpile model in \cite{AS2019}; and partial and inverse balayage methods are valuable tools for the study of Hele-Shaw flows and free boundary problems \cite{GS1994, G1997, G2004}.

In this paper, we consider the behavior of the balayage $\nu = \mathrm{Bal}(\mu,\partial \Omega)$ near a point $z_0 \in \partial \Omega$ at which $\Omega$ has a cusp. Assuming that the density of the non-negative measure $\mu$ on $\Omega$ vanishes or grows at a certain power rate at $z_0$, we determine the asymptotic behavior of $\nu$ near $z_0$. More precisely, if $d\mu(z) \asymp |z-z_{0}|^{2b-2}d^{2}z$ as $z\to z_0$ for some $b > 0$, then we determine the leading order term of $\nu(\partial \Omega \cap B_r(z_0))$ as $r \to 0$, where $B_r(z_0)$ is the open disk of radius $r$ centered at $z_0$. It turns out that the leading order term of $\nu(\partial \Omega \cap B_r(z_0))$ as $r\to 0$ is {\it universal} in the sense that it only depends on $b$ and on the geometry of the cusp (more precisely, it only depends on $b$ and on the order and coefficient of tangency of the cusp; these notions are defined in Section \ref{section:def of a cusp} below). In particular, the vanishing rate of $\nu$ is independent of the global structure of $\Omega$ and of the restriction of $\mu$ to $\Omega \setminus B_r(z_0)$ for any fixed $r>0$. 

We also treat the case when $\Omega$ has an arbitrary finite number of corners and an arbitrary finite number of cusps at the same point $z_0 \in \partial \Omega$. If there are only cusps at $z_0$, then it turns out that the cusp with the lowest order of tangency determines the vanishing rate of $\nu(\partial \Omega \cap B_r(z_0))$. If there are also corners at $z_0$, then the vanishing rate is determined by the corner with the largest opening angle. Since the ``width'' of a cusp increases with decreasing order of tangency, this means that corners and cusps contribute to the balayage measure near $z_0$ according to their width. 

The present work is a continuation of \cite{CL Corner}, which deals with the case where $\Omega$ has a corner of opening angle $\pi \alpha$, $0 < \alpha \leq 2$, at $z_{0}$. The results of \cite{CL Corner} with $\alpha = 2$ cover the case of an inward-pointing cusp, i.e., inward-pointing cusps are included as special cases of corners. In this paper, we consider cusps pointing outwards in the sense that the interior angle of $\Omega$ vanishes at the cusp (as opposed to inward-pointing cusps for which the interior angle of $\Omega$ equals $2\pi$).
Precise definitions of the classes of corners and cusps that our results apply to are given in Section \ref{mainsec}.

The analysis of cusps presents a number of differences compared to that of corners. For example, for any fixed $z \in \Omega$, the harmonic measure $\omega(z, \partial \Omega \cap B_{r}(z_0), \Omega)$ has exponential decay as $r \to 0$ in the case of a cusp, whereas it has only polynomial decay in the case of a corner. In view of (\ref{nuharmonicmeasure}), one might expect that this would lead to a much faster decay rate for $\nu(\partial \Omega \cap B_r(z_0))$ in the case of a cusp. However, there are other effects that also need to be taken into account: (1) if $|z - z_0| \lesssim r$, then $\omega(z, \partial \Omega \cap B_{r}(z_0), \Omega)$ can be quite large (i.e. close to $1$) for a cusp in the small $r$ limit, and (2) the integral in (\ref{nuharmonicmeasure}) runs over a smaller area near $z_0$ if $z_0$ is a cusp.  Taken together, these effects imply that the vanishing rate of $\nu(\partial \Omega \cap B_r(z_0))$ at a cusp is only smaller by a factor $r^d$, where $d>0$ is the order of tangency of the cusp, compared with the case of a corner with a small opening angle, see Theorems \ref{mainth} and \ref{mainth2}.

\subsection{Balayage of the uniform measure on a tacnodal region}
There are few domains with cusps for which the balayage measure $\nu$ in (\ref{nuharmonicmeasure}) can be computed explicitly. However, if $\mu$ is the uniform (Lebesgue) measure and $\Omega$ is the tacnodal region between two osculating circles, we are able to determine $\nu$ exactly, see Theorem \ref{thm:EG tacnode nu}. Such a region has two cusps with order of tangency $1$ at the point $z_0$ where the circles intersect.
By extracting the small $r$ behavior of $\nu(\partial \Omega \cap B_r(z_0))$ from the obtained explicit formula, we verify in Remark \ref{remark:check the main thm} the obtained general asymptotic formulas for $\nu(\partial \Omega \cap B_r(z_0))$ in this special case.

\subsection{Two-dimensional Coulomb gases}
As another application of Theorem \ref{thm:EG tacnode nu}, we consider a Coulomb gas model with $n$ points in the presence of an external potential. We choose the external potential corresponding to the elliptic Ginibre point process. To leading order, the points of this process distribute themselves uniformly on the region $S$ enclosed by an ellipse as $n \to \infty$. If one imposes the hard wall constraint that no points are allowed to lie in a specified subset $\Omega$ of $S$, then the probability that no points lie in $\Omega$ can be expressed in the large $n$ limit in terms of the balayage measure $\nu=\mathrm{Bal}(\mu|_{\Omega},\partial \Omega)$ where $\mu$ is the uniform measure \cite{A2018, AR2017, C2023}. Using Theorem \ref{thm:EG tacnode nu}, we compute the large $n$ behavior of the probability that no points lie on a tacnodal region between two osculating circles.

\subsection{Organization of the paper}
Our four main theorems are stated in Section \ref{mainsec}. Theorem \ref{mainth} gives the leading term of the balayage near a single cusp, while Theorem \ref{mainth2} treats the case of multiple corners and cusps at the same point. 
Theorem \ref{thm:EG tacnode nu} gives an explicit expression for $\nu=\mathrm{Bal}(\mu|_{\Omega},\partial \Omega)$ where $\Omega$ is the tacnodal region between two osculating circles and $\mu$ is the uniform measure. 
Theorem \ref{thm:EG tacnode C} gives the probability that the elliptic Ginibre point process with $n$ points has no points in a tacnodal region in the large $n$ limit.
The proofs of the four main theorems are presented in Sections \ref{proofsec}--\ref{holesec}, respectively.

\section{Main results}\label{mainsec}
Before stating our main results, we need to introduce some notation. In particular, we need to specify what we mean by a cusp. 
We will first introduce the concept of an analytic cusp with small perturbation of angles as defined in \cite{K2010}. Our main results will apply to cusps which locally are images of analytic cusps with small perturbation of angles under power maps of the form $z \mapsto z_0 + (z-z_0)^p$, $p >0$. 

\subsection{Definition of a cusp}\label{section:def of a cusp}
Let $B_r(z) \subset \C$ denote the open disk of radius $r$ centered at $z$. Let $\C^* = \C \cup \{\infty\}$ denote the extended complex plane. If $\Omega$ is an open connected subset of $\C^*$ such that $\partial \Omega$ is a finite union of pairwise disjoint Jordan curves, then we say that $\Omega$ is a {\it finitely connected Jordan domain}.
A {\it regular analytic arc} in $\C^*$ is the image of a closed interval $I\subset \R$ under an analytic function $w:I\to \C^*$ such that $w'(t)\neq 0$ for all $t\in I$.

\begin{figure}
\begin{center}
\begin{tikzpicture}[scale = 1.5]

  \coordinate (c1) at (1, 0.21);
  \coordinate (c2) at (3.5, -1);
  \coordinate (c3) at (5.5, 1.1);
  \coordinate (c4) at (3, 1.7);
  \coordinate (c5) at (1.1, 1.8);
  \coordinate (c6) at (.8, 0.55);
  
  \coordinate (e1) at (70:1.1);
  \coordinate (e2) at (0, 2);
  \coordinate (e3) at (100:3);
  \coordinate (e4) at (150:3);
  \coordinate (e5) at (120:2);
  \coordinate (e6) at (80:1.1);

 \coordinate (a1) at (3.1, -0.6);
  \coordinate (a2) at (3.2, -0.3);
  \coordinate (a3) at (3.6, -0.1);
  \coordinate (a4) at (3.4, -0.3);

 \coordinate (b1) at (4, 0.6);
  \coordinate (b2) at (4.1, 0.9);
  \coordinate (b3) at (3.6, 0.8);
  \coordinate (b4) at (3.3, 0.4);

 \coordinate (d1) at (2, 1.4);
  \coordinate (d2) at (2.3, 0.5);
  \coordinate (d3) at (2.9, 0.8);

\coordinate (origin) at (70:0.4);

  
  \draw[black, thick, fill=cyan!8] plot [smooth] coordinates {(origin) (c1) (c2) (c3) (c4) (e3) (e4) (e5) (e2) (c5) (c6) (origin)};
  
\draw[dashed] (origin) circle (0.9);  


\draw[cyan, line width=0.5mm] (origin)+(-11:0.9) arc (-11:29:0.9);
\node[cyan] at ($(origin)+(9:1.2)$){\footnotesize $r \Theta(r)$};


  \draw[black, thick, fill=white] plot [smooth cycle] coordinates {(a1) (a2) (a3) (a4)};
  \draw[black, thick, fill=white] plot [smooth cycle] coordinates {(b1) (b2) (b3) (b4)};
  \draw[black, thick, fill=white] plot [smooth cycle] coordinates {(d1) (d2) (d3)};
  

\draw (origin) -- +(235:0.9);
\node at (-0.15,0.1){\footnotesize $r$};


\draw[fill] (origin) circle (1pt);
\node at ($(origin)+(-0.16,0)$) {\footnotesize $z_0$};

\coordinate (z0) at (2.5,0.15);


  \begin{scope}
    \clip (origin) circle (1.2);
  \draw[red, thick, line width=0.5mm] plot [smooth] coordinates {(origin) (c1) (c2) (c3) (c4) (e3) };
  \end{scope}
\node[red] at ($(origin)+(-22:0.51)$){\footnotesize $C_{+}$};
\node[red] at ($(origin)+(22:0.48)$){\footnotesize $C_{-}$};
\begin{scope}
    \clip (origin) circle (1.2);
  \draw[red, thick, line width=0.5mm] plot [smooth] coordinates {(origin) (c6) (c5) (e2) };
  \end{scope}

  \node at (4.8,0.9){$\Omega$};
 
\end{tikzpicture} 
 \caption{\label{harmonicfigure} Illustration of a finitely connected Jordan domain $\Omega$ with an analytic cusp at $z_{0}$.}
 \end{center}
\end{figure}
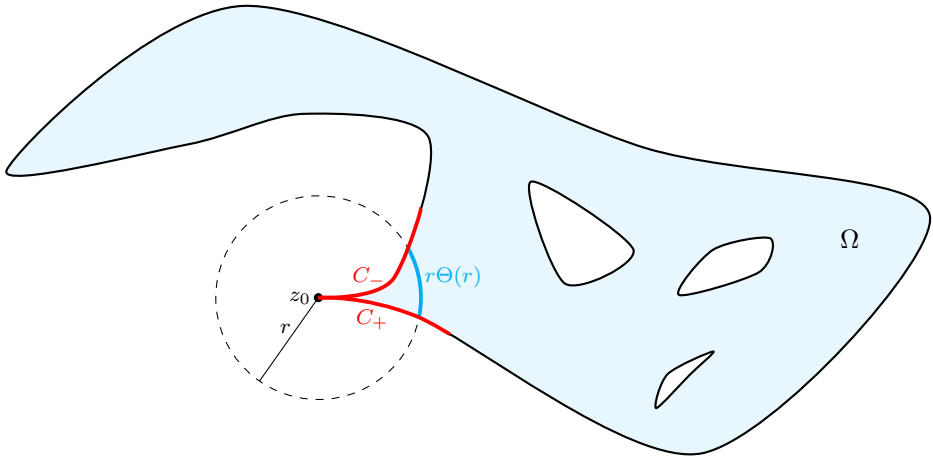

A finitely connected Jordan domain $\Omega$ has an {\it analytic cusp} at $z_0 \in \partial \Omega \cap \C$ if $\partial \Omega$ near $z_0$ consists of two regular analytic arcs $C_\pm \subset \partial \Omega$ such that the interior angle of $\Omega$ at $z_0$ vanishes, see Figure \ref{harmonicfigure}.
This means that there is an $\epsilon > 0$ and analytic functions $w_\pm:B_{2\epsilon}(0) \to \C$ such that $w_\pm(0) = z_0$, $w_+'(0) = w_-'(0) \neq 0$, and such that the boundary of $\Omega$ near $z_0$ is given by $C_+ \cup C_-$ with $C_\pm = w_\pm([0,\epsilon])$. 
We choose $C_\pm$ so that if they are oriented outwards from $z_0$, then $\Omega$ lies to the left of $C_+$ and to the right of $C_-$. After applying a rigid rotation if necessary, we may assume that the cusp is {\it tangent to $\R_{\geq 0}$} in the sense that $w_+'(0) = w_-'(0) > 0$. Furthermore, shrinking $\epsilon > 0$ if necessary, we may after a reparameterization assume that $w_\pm(r) = r e^{i \theta_\pm(r)}$ for $r \in [0, \epsilon]$, where $\theta_\pm(r)$ are real analytic functions of $r$ in a neighborhood of $r=0$ such that $\theta_\pm(0) = 0$, see \cite[p. 37]{K2010}.

If $h(r) = \sum_{n=0}^\infty h_n r^n$ is a non-vanishing power series, we define $\ord(h)$ as the smallest value of $n$ for which $h_n$ is nonzero; if $h(r) \equiv 0$, we set $\ord(0) = \infty$. 
The function $\Theta(r) := \theta_+(r) - \theta_-(r)$ admits a power series representation that converges near $r = 0$, 
$$\Theta(r) = \sum_{n=1}^\infty a_n r^n,$$ 
and the {\it order of tangency} of the cusp is the integer $d := \ord(\Theta(r)) \geq 1$. The first nonzero coefficient, $a_d$, is the {\it coefficient of tangency} of the cusp. Note that $\Theta(r) = a_d r^d + \bigO(r^{d+1})$ as $r \to 0$. Following \cite{K2010}, we say that the cusp has {\it small perturbation of angles} if $\min\{\ord(\theta_+(r)), \ord(\theta_-(r))\} = d$ and the coefficients $a_{d+1}, \dots, a_{2d}$ are all zero, i.e., $\Theta(r) = a_d r^d + \bigO(r^{2d+1})$ as $r \to 0$.

\begin{example}
Let $d \geq 1$ be an integer, $a > 0$, and $\rho \in (0, (\pi/(2a))^{1/d})$. Then $\Omega = \{z \in \C : 0 < |z| < \rho, 0 < \arg{z} < a |z|^d\}$ has an analytic cusp at $0$ with small perturbation of angles; its order of tangency is $d$ and its coefficient of tangency is $a$. 
Similarly, the domain in the right half of $\C \simeq \R^2$ between the $x$-axis and the graph of the function $y = a x^{d+1}$ has an analytic cusp at $0$ with small perturbation of angles with order of tangency $d$ and coefficient of tangency $a$. 
\end{example}

We say that a finitely connected Jordan domain $\Omega$ has a {\it cusp} at $z_0 \in \partial \Omega \cap \C$ if there is a radius $r > 0$ and a $p > 0$ such that the map $\varphi(z) = z_0 + (z-z_0)^p$ is bijective on $\Omega \cap B_r(z_0)$ and its image $\varphi(\Omega \cap B_r(z_0))$ has an analytic cusp at $z_0$ with small perturbation of angles. This definition of a cusp implies that $\Theta(r) = a r^d + \bigO(r^{2d+p})$ as $r \to 0$ for some $a>0$ and $d>0$, where $r\Theta(r)$ is the length of $\Omega \cap \partial B_r(z_0)$. We call $d$ and $a$ the {\it order of tangency} and the \textit{coefficient of tangency} of the cusp, respectively.

Whereas the order of tangency of an analytic cusp always is an integer, the order of tangency of a cusp can be any positive real number.

\begin{example}
Let $d > 0$, $a > 0$, and $\rho \in (0, (\pi/(2a))^{1/d})$. Then $\Omega = \{z \in \C : 0 < |z| < \rho, 0 < \arg{z} < a |z|^d\}$ has cusp at $0$ with order of tangency $d$ and coefficient of tangency $a$.
Similarly, the domain in the right half of $\C \simeq \R^2$ between the $x$-axis and the graph of the function $y = a x^{d+1}$ has a cusp at $0$ with order of tangency $d$ and coefficient of tangency $a$. 
\end{example}

\subsection{A single cusp}
We first state our main result in the case of a single cusp at $z_0$, see Figure \ref{harmonicfigure}. The result will then be generalized in Section \ref{multiplesubsec} to the case of arbitrary finite numbers of corners and/or cusps at $z_0$. 

We write 
\begin{align}\label{muc0o1}
d\mu(z) = (1 +o(1)) |z-z_{0}|^{2b-2}d^{2}z \qquad \text{as $z\to z_0 \in \partial \Omega$}
\end{align}
if for every $\epsilon > 0$, there exists $\rho_0 >0$ so that
\begin{align}\label{muc0o1eps}
\bigg|\mu(A) - \int_A |z-z_{0}|^{2b-2}d^{2}z\bigg| \leq \epsilon \int_A |z-z_{0}|^{2b-2}d^{2}z
\end{align}
holds for all measurable subsets $A$ of $\Omega \cap B_{\rho_0}(z_0)$.
Similarly, we write
$$d\mu(z) \asymp |z-z_{0}|^{2b-2}d^{2}z \qquad \text{as $z\to z_0 \in \partial \Omega$}$$
if there exist $\rho_0 >0$ and $c_1, c_2 > 0$ such that
\begin{align}
c_1 \int_A |z-z_{0}|^{2b-2}d^{2}z \leq \mu(A) \leq c_2 \int_A |z-z_{0}|^{2b -2}d^{2}z
\end{align}
for all measurable sets $A \subset \Omega \cap B_{\rho_0}(z_0)$, where
$d^2z = dxdy$ is the Lebesgue measure on $\C$. 
Given two positive functions $f$ and $g$, we also write 
$$f(r) \asymp g(r) \qquad \text{as $r \to 0$}$$
if there exists an $r_0 > 0$ and constants $c_1, c_2 > 0$ such that 
$$c_1 g(r) \leq f(r) \leq c_2 g(r) \qquad \text{for all $r \in (0, r_0)$}.$$ 

In the case of a single cusp, our main result is the following. We assume that $\omega(\cdot;E,\Omega)$ is $\mu$-measurable for any Borel subset $E$ of $\partial \Omega$, so that the integral in \eqref{nuharmonicmeasure} is well-defined.

\begin{theorem}[A single cusp]\label{mainth}
Let $\Omega$ be a finitely connected Jordan domain in $\C^* = \C \cup \{\infty\}$. Suppose $\Omega$ has a cusp with order of tangency $d > 0$ and coefficient of tangency $a>0$ at a point $z_0 \in \partial \Omega \cap \C$. Let $\mu$ be a non-negative measure of finite total mass on $\Omega$ such that $d\mu(z) = (1+o(1)) |z-z_{0}|^{2b-2}d^{2}z$ as $z\to z_0$ for some $b > 0$. Then, the balayage $\nu := \mathrm{Bal}(\mu,\partial \Omega)$ of $\mu$ onto $\partial \Omega$ obeys the following estimate as $r \to  0$:
\begin{align}\label{nuestimate}
\nu(\partial \Omega \cap B_r(z_0)) =(1+o(1))\frac{a}{2b+d} r^{2b+d}.
\end{align}
In particular, if $d\mu(z) \asymp |z-z_{0}|^{2b-2}d^{2}z$ as $z\to z_0$, then
\begin{align*}
\nu(\partial \Omega \cap B_r(z_0)) \asymp r^{2b+d} \qquad \mbox{as } r\to 0.
\end{align*}
\end{theorem}

The proof of Theorem \ref{mainth} is presented in Section \ref{proofsec}.

\begin{remark}
The asymptotic formula \eqref{nuestimate} can be rewritten as follows: for every $\epsilon >0$, there exists an $r_{0} >0$ such that, for all $r\in (0,r_{0}]$, 
\begin{align}\label{nuestimate eps}
(1-\epsilon)\frac{a}{2b+d} r^{2b+d}\leq \nu(\partial \Omega \cap B_r(z_0)) \leq (1+\epsilon)\frac{a}{2b+d} r^{2b+d}.
\end{align}
We will prove the upper and lower bounds in \eqref{nuestimate eps} separately. The assumptions we have imposed on the cusp in Theorem \ref{mainth} (i.e., that $\Omega$ locally is the image under a power map of an analytic cusp with small perturbation of angles) are needed to prove the lower bound $\frac{a(1-\epsilon)}{2b+d}  r^{2b+d} \leq \nu(\partial \Omega \cap B_r(z_0))$ in (\ref{nuestimate eps}). The upper bound in (\ref{nuestimate eps}), $\nu(\partial \Omega \cap B_r(z_0)) \leq \frac{a(1+\epsilon)}{2b+d} r^{2b+d}$, holds even without these assumptions as long as $\Theta(r) \leq (1+o(1)) a r^d$ as $r\to 0$; see the proof in Section \ref{proofsec}.
\end{remark}

\begin{remark}(Consistency check.)
Suppose that $\Omega_{1}$ is a Jordan domain with a corner of opening $\pi \alpha$ at $z_{0}$, and that $\Omega_{2}$ is a Jordan domain with a cusp of order $d$ and coefficient of tangency $a$ at $z_{0}$. In particular, as $r\to 0$,
\begin{align*}
\Theta_{1}(r) := \frac{1}{r} |\Omega_{1} \cap \partial B_{r}(z_{0})| = \pi \alpha(1+o(1)), \qquad \Theta_{2}(r) := \frac{1}{r}  |\Omega_{2} \cap \partial B_{r}(z_{0})| = ar^{d}(1+o(1)).
\end{align*}
Let $\mu$ be a non-negative measure of finite total mass on $\Omega_{1}\cup \Omega_{2}$ such that $d\mu(z) = (1+o(1)) |z-z_{0}|^{2b-2}d^{2}z$ as $z\to z_0$ for some $b > 0$. Note that the leading term in \eqref{nuestimate} is universal in that it only depends on $d$, $a$, and $b$. Hence, if $\alpha$ and $d$ are small (but fixed) and $a=\pi \alpha$, one expects 
\begin{align}\label{lol22}
\nu_{1}(\partial \Omega_{1} \cap B_r(z_0)) \approx \nu_{2}(\partial \Omega_{2} \cap B_r(z_0)) \qquad \mbox{for small } r>0,
\end{align}
where $\nu_{j} := \mathrm{Bal}(\mu|_{\Omega_{j}},\partial \Omega_{j})$, $j=1,2$. We verify here that this is indeed the case. By \cite[Theorem 2.1]{CL Corner}, for any $\epsilon >0$,  
\begin{align*}
(1-\epsilon) \frac{\tan(\pi \alpha b)}{2b^2} r^{2b} & \leq \nu_{1}(\partial \Omega_{1} \cap B_r(z_0)) \leq (1+\epsilon) \frac{\pi \alpha}{2b} \bigg(1  + \frac{16 b}{\pi (\frac{1}{\alpha} - 2b)}\bigg)  r^{2b}
\end{align*}
holds for all sufficiently small $r>0$. As $\alpha \to 0$, 
\begin{align*}
\frac{\tan(\pi \alpha b)}{2b^2} = \frac{\pi \alpha}{2b} + \bigO(\alpha^{3}), \qquad \frac{\pi \alpha}{2b} \bigg(1  + \frac{16 b}{\pi (\frac{1}{\alpha} - 2b)}\bigg) = \frac{\pi \alpha}{2b} + \bigO(\alpha^{2}).
\end{align*}
Hence, if $\alpha \approx 0$, then 
\begin{align}\label{lol20}
\nu_{1}(\partial \Omega_{1} \cap B_r(z_0)) \approx \frac{\pi\alpha}{2b}r^{2b} \qquad \mbox{for small } r>0.
\end{align}
On the other hand, if $d \approx 0$ and $a=\pi \alpha$, Theorem \ref{mainth} implies that
\begin{align}\label{lol21}
\nu_{2}(\partial \Omega_{2} \cap B_r(z_0)) \approx \frac{\pi\alpha}{2b}r^{2b} \qquad \mbox{for small } r>0.
\end{align}
Taken together, \eqref{lol20} and \eqref{lol21} confirm the expectation \eqref{lol22}. 
\end{remark}

\subsection{Multiple corners and cusps}\label{multiplesubsec}
Our second theorem treats the case when $\Omega$ has $m_1$ corners and  $m_2$ cusps at $z_0$, see Figure \ref{cornerfigure}.

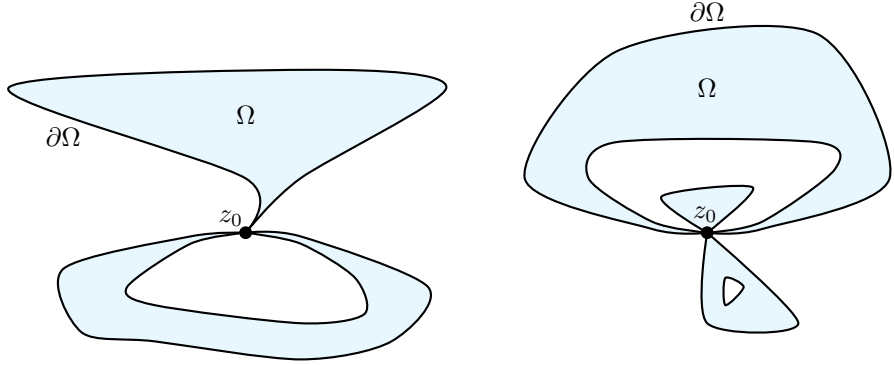
\begin{figure}
\begin{center}
\begin{tikzpicture}[master,scale = 2.4]
  \coordinate (origin) at (0, 0);

  \coordinate (z0) at (0,0);
  \coordinate (c2) at (0.3, 0.3);
  \coordinate (c3) at (1.1, 0.8);
  \coordinate (c4) at (0.4, 0.9);
  \coordinate (c5) at (-1.3, 0.8);
  \coordinate (c6) at (0, 0.3);
  
  \coordinate (d2) at (0.3, -0.02);
  \coordinate (d3) at (1, -0.3);
  \coordinate (d3bis) at (0.8, -0.6);
  \coordinate (d4) at (0.3, -0.7);
  \coordinate (d5) at (-0.5, -0.6);
  \coordinate (d5bis) at (-0.9, -0.55);
  \coordinate (d6) at (-1.0, -0.2);
  \coordinate (d7) at (-0.3, -0.02);
  
  \coordinate (e2) at (0.3, -0.06);
  \coordinate (e3) at (0.6, -0.25);
  \coordinate (e3bis) at (0.65, -0.45);
  \coordinate (e4) at (0.3, -0.5);
  \coordinate (e5) at (-0.5, -0.4);
  \coordinate (e6) at (-0.65, -0.3);
  \coordinate (e7) at (-0.3, -0.06);

  \draw[black, thick, fill=cyan!8] plot [smooth] coordinates {(z0) (c2) (c3) (c4) (c5) (c6) (z0)};
  
  \draw[black, thick, fill=cyan!8] plot [smooth] coordinates {(z0) (d2) (d3) (d3bis) (d4) (d5) (d5bis) (d6) (d7) (z0)};
  
  \draw[black, thick, fill=white!] plot [smooth] coordinates {(z0) (e2) (e3) (e3bis) (e4) (e5) (e6) (e7) (z0)};





%
\node at (0,0.65){$\Omega$};
\node at (-1,0.53){$\partial\Omega$};

\fill (origin) circle (1pt);
\node at (-0.08,0.08){$z_0$};

\end{tikzpicture} 
\begin{tikzpicture}[slave, scale = 2.4]
  \coordinate (origin) at (0, 0);
  
  \coordinate (d2) at (0.3, 0.02);
  \coordinate (d3) at (1, 0.3);
  \coordinate (d3bis) at (0.8, 0.6);
  \coordinate (d4) at (0.6, 1.1);
  \coordinate (d5) at (-0.5, 1);
  \coordinate (d6) at (-1, 0.3);
  \coordinate (d7) at (-0.3, 0.02);
  
  \coordinate (e2) at (0.3, 0.06);
  \coordinate (e3) at (0.7, 0.35);
  \coordinate (e4) at (0.6, 0.5);
  \coordinate (e5) at (-0.5, 0.5);
  \coordinate (e6) at (-0.65, 0.3);
  \coordinate (e7) at (-0.3, 0.06);
  
  \coordinate (f1) at (0.5, -0.5);
  \coordinate (f2) at (0, -0.5);
  
  \coordinate (g1) at (0.25, 0.25);
  \coordinate (g2) at (-0.250, 0.2);
  
  \coordinate (h1) at (0.1, -0.25);
  \coordinate (h2) at (0.2, -0.3);
  \coordinate (h3) at (0.1, -0.4);

  
  \draw[black, thick, fill=cyan!8] plot [smooth] coordinates {(z0) (d2) (d3) (d4) (d5) (d6) (d7) (z0)};
  
  \draw[black, thick, fill=white!] plot [smooth] coordinates {(z0) (e2) (e3) (e4) (e5) (e6) (e7) (z0)};
  
  \draw[black, thick, fill=cyan!8] plot [smooth] coordinates {(z0) (f1) (f2) (z0)};
  
  \draw[black, thick, fill=cyan!8] plot [smooth] coordinates {(z0) (g1) (g2) (z0)};
  
  \draw[black, thick, fill=white!] plot [smooth cycle] coordinates {(h1) (h2) (h3)};





%
\node at (0,0.8){$\Omega$};
\node at (0,1.22){$\partial\Omega$};

\fill (origin) circle (1pt);
\node at (-0.01,0.1){$z_0$};

\end{tikzpicture} 
 \caption{\label{cornerfigure} Left: an open set $\Omega$ with $m_{1}=0$ corners and $m_{2}=3$ cusps at $z_{0}$. Right: an open set $\Omega$ with $m_{1}=2$ corners and $m_{2}=2$ cusps at $z_{0}$.
}
  \end{center}
\end{figure}

We define a corner as follows. The image of a closed interval $I \subset \R$ under an injective continuous map $I \to \C^*$ is a {\it Jordan arc in $\C^*$}. A Jordan arc $C \subset \C$ is of class $C^{1,\gamma}$, $0 < \gamma \leq 1$, if it has a parametrization $C: w(t)$, $0 \leq t \leq 1$, such that $w'(t)$ exists, is nonzero, and is H\"older continuous with exponent $\gamma$ on $[0,1]$.
A finitely connected Jordan domain $\Omega$ in $\C^*$ has a {\it corner} of opening $\pi \alpha$, $0 < \alpha \leq 2$, at $z_0 \in \partial \Omega \cap \C$ if there are Jordan arcs $C_\pm \subset \partial \Omega$ ending at $z_0$ and a $\phi \in \R$ such that
\begin{align}\label{arg def angle}
\arg(z - z_0) \to \begin{cases} 
\phi & \text{as $C_+ \ni z \to z_0$}, \\
\phi + \pi \alpha & \text{as $C_- \ni z \to z_0$}, 
\end{cases}
\end{align}
and such that $\Omega$ lies to the left of $C_+$ and to the right of $C_-$ if $C_\pm$ are oriented outwards from $z_0$.
 In \eqref{arg def angle}, the branch is such that $\arg(\cdot-z_{0})$ is continuous in $(\bar{\Omega} \cap B_{\rho_{0}}(z_{0}))\setminus \{z_{0}\}$ for some small enough $\rho_{0}>0$. 
The corner is {\it H\"older-$C^1$} if the Jordan arcs $C_\pm$ can be chosen to be of class $C^{1,\gamma}$ for some $\gamma>0$.

\begin{theorem}[Multiple corners and/or cusps at the same point]\label{mainth2}
Let $\Omega$ be an open subset of $\C^*$, let $z_0 \in \partial \Omega \cap \C$, and let $m_1,m_2 \geq 0$ be integers such that $m := m_1 + m_2 \geq 1$. Suppose there is a radius $\rho_0 > 0$ such that

\begin{enumerate}[$(i)$]
\item the open set $\Omega \cap B_{\rho_0}(z_0)$ has $m$ connected components $\{U_j\}_1^m$ satisfying $\bar{U}_j \cap \bar{U}_k = \{z_0\}$ whenever $j \neq k$, and such that for each $j = 1, \dots, m_1$, $U_j$ has a H\"older-$C^1$ corner at $z_0$ of opening $\pi \alpha_j \in (0,2\pi]$, while for each $j = m_1 + 1, \dots, m$,  $U_j$ has a cusp at $z_0$ with order of tangency $d_j> 0$ and coefficient of tangency $a_{j}>0$.

\item $\Omega \setminus \overline{B_{\rho_0}(z_0)}$ is a disjoint union of finitely connected Jordan domains.
\end{enumerate}
Let
$$\alpha := \displaystyle{\max_{1\leq j \leq m_1}} \alpha_j \quad \text{and} \quad 
d := \displaystyle{\min_{m_1 +1\leq j \leq m}} d_j.$$
Let $\mu$ be a non-negative measure of finite total mass on $\Omega$ such that $d\mu(z) = (1+o(1))  |z-z_{0}|^{2b-2}d^{2}z$ as $z\to z_0$ for some $b > 0$. 
For every $\epsilon > 0$, the balayage $\nu := \mathrm{Bal}(\mu,\partial \Omega)$ of $\mu$ onto $\partial \Omega$ obeys the following inequalities for all sufficiently small $r > 0$:
\begin{align}
(1-\epsilon)\frac{\sum_{j \in \mathcal{D}}a_{j}}{2b+d} r^{2b+d} & \leq \nu(\partial \Omega \cap B_r(z_0)) \leq (1+\epsilon)\frac{\sum_{j \in \mathcal{D}}a_{j}}{2b+d} r^{2b+d} & & \hspace{-1cm}  \mbox{if } m_{1}=0, \nonumber \\
\nonumber
(1-\epsilon) \sum_{j=1}^{m_{1}} \frac{\tan(\pi \alpha_j b)}{2b^2} r^{2b} & \leq \nu(\partial \Omega \cap B_r(z_0)) \leq (1+\epsilon) \sum_{j=1}^{m_{1}} \frac{\pi \alpha_j}{2b} \bigg(1  + \frac{16 b}{\pi (\frac{1}{\alpha_j} - 2b)}\bigg)  r^{2b} \\
&  && \hspace{-2.2cm} \text{if $m_{1}\geq 1$ and $2b < \frac{1}{\alpha}$}, \nonumber 
	\\ \nonumber 
m_{\alpha} (1-\epsilon) \frac{2}{\pi b} r^{2b} \log(\tfrac{1}{r})
& \leq
 \nu(\partial \Omega \cap B_r(z_0))
\leq m_{\alpha} (1+\epsilon) \frac{4}{b} r^{2b} \log(\tfrac{1}{r}) & & \hspace{-2.2cm} \text{if } m_{1}\geq 1 \mbox{ and } 2b= \frac{1}{\alpha}, 	\\ \label{numultiplebounds}
c r^{\frac{1}{\alpha}} & \leq
 \nu(\partial \Omega \cap B_r(z_0))
\leq C r^{\frac{1}{\alpha}} & & \hspace{-2.2cm} \text{if } m_{1}\geq 1 \mbox{ and } 2b > \frac{1}{\alpha},
\end{align}
for some constants $c,C>0$, and where $\mathcal{D} = \{j \in \{m_{1}+1,\ldots,m\}: d_{j}=d\}$ and $m_{\alpha} = \#\{j \in \{1,\ldots, m_{1}\}:\alpha_{j}=\alpha\}$.

In particular, if $d\mu(z) \asymp |z-z_{0}|^{2b-2}d^{2}z$ as $z\to z_0$, then
\begin{align}\label{nuestimate2}
\nu(\partial \Omega \cap B_r(z_0)) \asymp \begin{cases}
r^{2b+d}  & \text{if $m_1=0$}, 
	\\
r^{2b} & \text{if $m_1\geq 1$ and $2b < \frac{1}{\alpha}$},
 	\\
r^{2b}\log \frac{1}{r} &  \text{if $m_1\geq 1$ and $2b= \frac{1}{\alpha}$}, 
	\\
r^{\frac{1}{\alpha}} &  \text{if $m_1\geq 1$ and $2b > \frac{1}{\alpha}$}.
\end{cases}
\end{align}
\end{theorem}

The proof of Theorem \ref{mainth2} is presented in Section \ref{proofsec2}.

\subsection{Balayage of the uniform measure on a tacnodal region}


\begin{figure}
\begin{center}
\hspace{-0.3cm}\begin{tikzpicture}[master]
\node at (0,0) {\includegraphics[width=4.8cm]{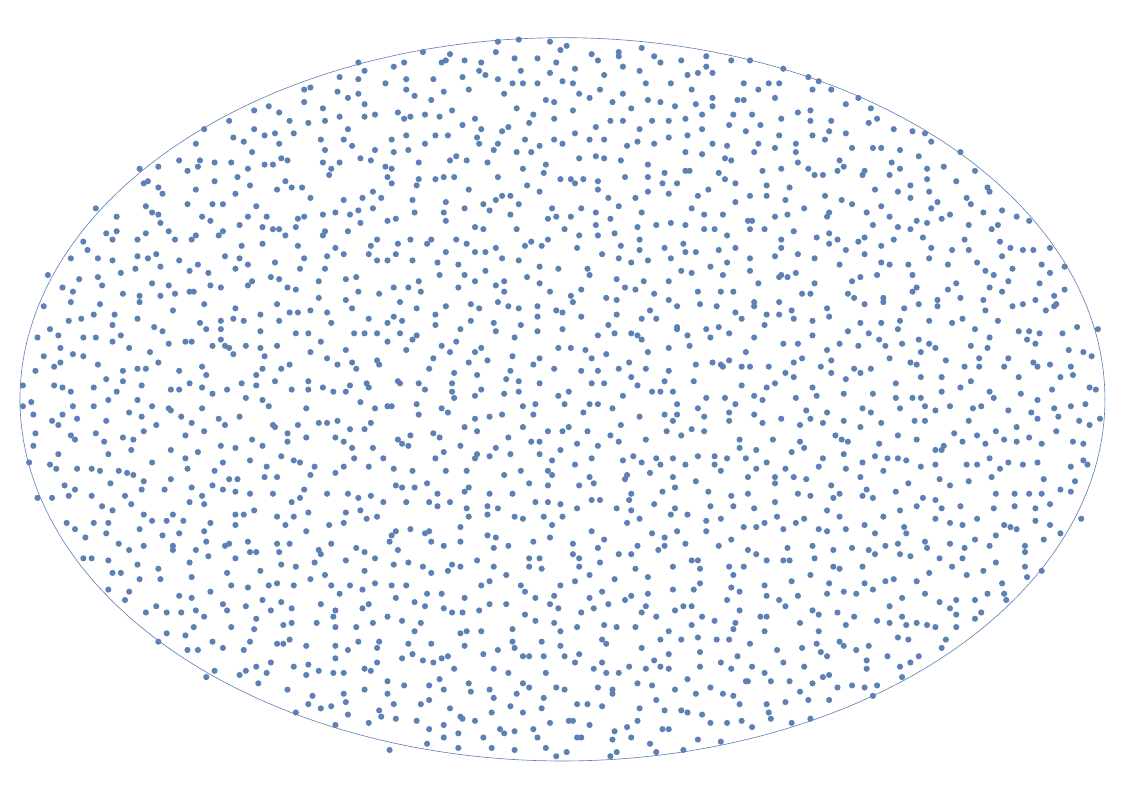}};
\end{tikzpicture}  \hspace{-0.4cm}
\begin{tikzpicture}[slave]
\node at (0,0) {\includegraphics[width=4.8cm]{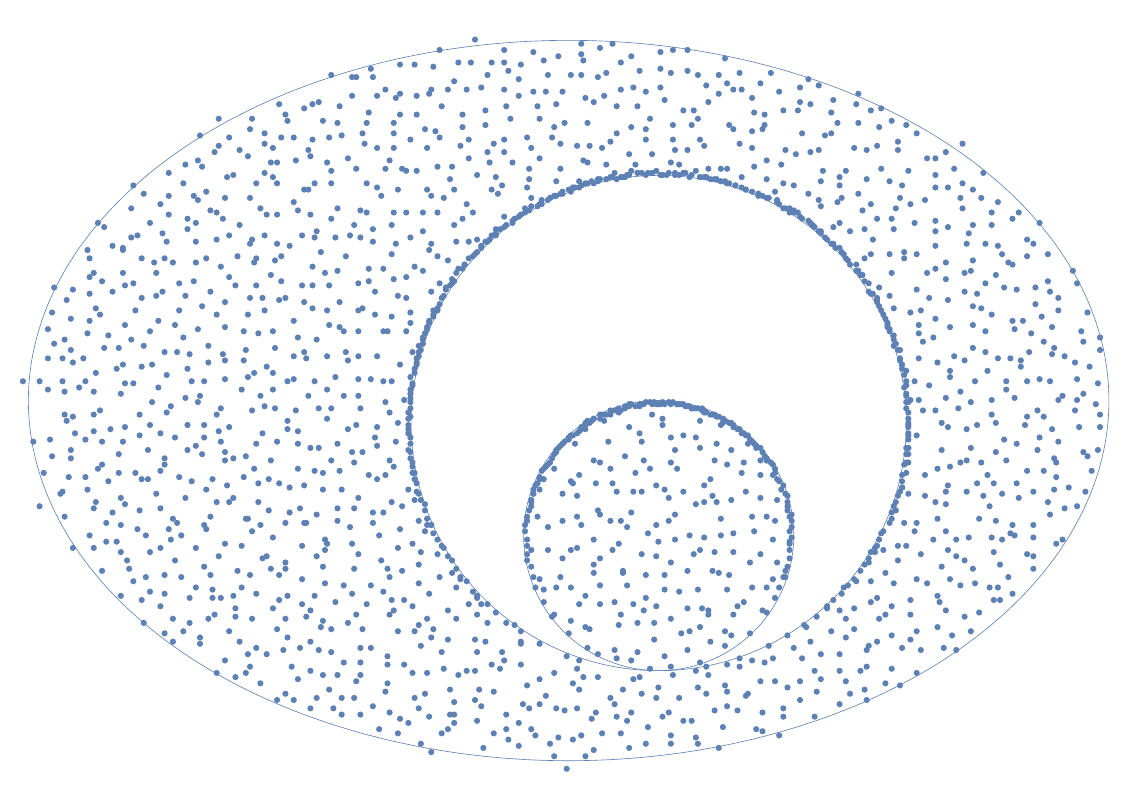}};
\end{tikzpicture} \hspace{-0.22cm}
\begin{tikzpicture}[slave]
\node at (0,0) {\includegraphics[width=4.8cm]{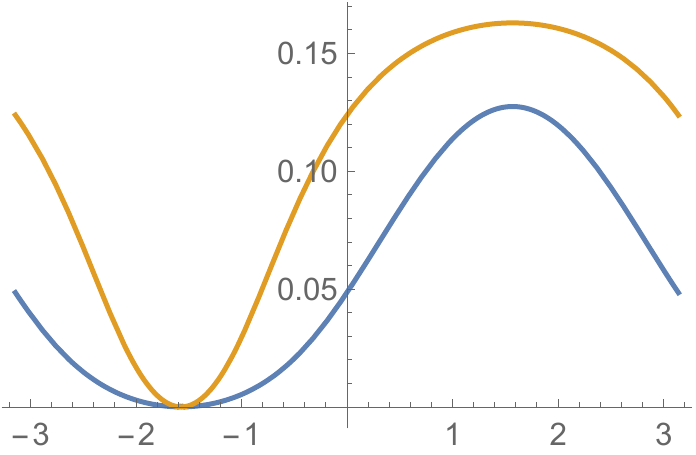}};

\end{tikzpicture}
\end{center}
\caption{\label{fig:density complement ellipse Ginibre} Left: the elliptic Ginibre point process with $\tau=0.2$ and $n=2000$. Middle: the same point process, but conditioned on $\# \{z_{j}\in z_0 + e^{i\theta_0}\Omega_{a,c}\} = 0$, where  $a=0.3$, $c = 0.55$, $\theta_{0}=0$ and $z_{0} = 0.2-0.6 i$. Right: the normalized density $\theta \mapsto \frac{d\nu(z_{0}+ai+ae^{i\theta})/d\theta}{\nu(\partial (z_0 + e^{i\theta_0}\Omega_{a,c}))}$ (blue) and $\theta \mapsto \frac{d\nu(z_{0}+ci+ce^{i\theta})/d\theta}{\nu(\partial (z_0 + e^{i\theta_0}\Omega_{a,c}))}$ (orange), i.e., the blue and orange curves show the normalized density of the balayage measure on the inner and outer circles of $\partial \Omega_{a,c}$, respectively. In accordance with Theorems \ref{mainth2} and \ref{thm:EG tacnode nu}, the density vanishes quadratically at the point where the circles intersect (corresponding to $\theta = -\pi/2$).}
\end{figure}

Let $\Omega_{a,c}$ be the the tacnodal region between two osculating circles of radii $0<a<c$ defined by
\begin{align}\label{Omegaacdef}
\Omega_{a,c} := \{z:|z-ci|<c\} \setminus \{z:|z-ai|\leq a\}.
\end{align}
Our next result gives an explicit expression for the balayage of the uniform measure on $\Omega_{a,c}$, see Figure \ref{fig:density complement ellipse Ginibre} (right).
The proof is presented in Section \ref{uniformsec}.

\begin{theorem}[Balayage of the uniform measure on a tacnodal region]\label{thm:EG tacnode nu}
Let $0<a<c$ and let $\nu=\mathrm{Bal}(\mu|_{\Omega_{a,c}},\partial \Omega_{a,c})$, where $d\mu(z) = \frac{d^{2}z}{\pi}$. 
For $k\in \Z$ and $\theta \in (-\frac{\pi}{2},\frac{3\pi}{2})\setminus\{\frac{\pi}{2}\}$, define
\begin{align}
& f_{k}(\theta) = \frac{e^{i\theta}}{(ic-(i+e^{i\theta})(c-a)k)^{2}}, & & g_{k}(\theta) = \frac{-e^{i\theta}}{(ia+(i+e^{i\theta})(c-a)k)^{2}}, \nonumber \\
& p_{k}(\theta) = \frac{a(i+e^{i\theta})}{1+ki(i+e^{i\theta})(1-\frac{a}{c})}, & & q_{k}(\theta) = \frac{c(i+e^{i\theta})}{1+ki(i+e^{i\theta})(1-\frac{c}{a})}. \label{def of pk and qk}
\end{align}
Then the balayage $\nu$ is given for $z=ai + ae^{i\theta}$, $\theta \in (-\frac{\pi}{2},\frac{3\pi}{2})\setminus\{\frac{\pi}{2}\}$ by 
\begin{align}
& \frac{d\nu(ai+ae^{i\theta})}{d\theta} = \frac{c}{(c-a)(1+\sin \theta)} \bigg\{ \frac{\frac{c^{2}}{2}\sin \big(\frac{c-2a}{c-a}\pi\big)}{\cosh\big( \frac{c\pi \cos \theta}{(c-a)(1+\sin \theta)} \big) -\cos\big(\frac{c-2a}{c-a}\pi\big)} \nonumber \\
&- \frac{\frac{a^{2}}{2}\sin\big(\frac{c\pi}{c-a}\big)}{\cosh\big( \frac{c\pi \cos \theta}{(c-a)(1+\sin \theta)} \big) -\cos\big(\frac{c\pi}{c-a}\big)}
 + \frac{a^{2}c\pi}{c-a} \frac{\cosh\big( \frac{c\pi \cos \theta}{(c-a)(1+\sin \theta)} \big) \cos\big(\frac{c\pi}{c-a}\big)-1}{\big( \cosh\big( \frac{c\pi \cos \theta}{(c-a)(1+\sin \theta)} \big) -\cos\big(\frac{c\pi}{c-a}\big) \big)^{2}}  \bigg\} \nonumber \\
& + \re \bigg[ \frac{ac^{3}}{\pi} \sum_{k=-\infty}^{0} f_{k}(\theta) \frac{ip_{k}(\theta)}{p_{k}(\theta)-ci} + \frac{a^{2}c^{2}}{\pi} \sum_{k=1}^{+\infty} f_{k}(\theta) \frac{ip_{k}(\theta)}{p_{k}(\theta)-ai} \bigg] \label{nu on Ca}
\end{align}
and for $z=ci + ce^{i\theta}$, $\theta \in (-\frac{\pi}{2},\frac{3\pi}{2})\setminus\{\frac{\pi}{2}\}$ by
\begin{align}
& \frac{d\nu(ci+ce^{i\theta})}{d\theta} = \frac{a}{(c-a)(1+\sin \theta)} \bigg\{ \frac{\frac{c^{2}}{2}\sin \big(\frac{a\pi}{c-a}\big)}{\cosh\big( \frac{a\pi \cos \theta}{(c-a)(1+\sin \theta)} \big) -\cos\big(\frac{a\pi}{c-a}\big)} 
	\nonumber \\
& + \frac{\frac{a^{2}}{2}\sin\big(\frac{a\pi}{c-a}\big)}{\cosh\big( \frac{a\pi \cos \theta}{(c-a)(1+\sin \theta)} \big) -\cos\big(\frac{a\pi}{c-a}\big)} 
- \frac{a^{2}c\pi}{c-a} \frac{\cosh\big( \frac{a\pi \cos \theta}{(c-a)(1+\sin \theta)} \big) \cos\big(\frac{a\pi}{c-a}\big)-1}{\big( \cosh\big( \frac{a\pi \cos \theta}{(c-a)(1+\sin \theta)} \big) -\cos\big(\frac{a\pi}{c-a}\big) \big)^{2}} \bigg\} 
	\nonumber \\
& + \re \bigg[ \frac{a^{3}c}{\pi} \sum_{k=-\infty}^{0} g_{k}(\theta) \frac{iq_{k}(\theta)}{q_{k}(\theta)-ai} + \frac{a^{2}c^{2}}{\pi} \sum_{k=1}^{+\infty} g_{k}(\theta) \frac{iq_{k}(\theta)}{q_{k}(\theta)-ci} \bigg]. \label{nu on Cc}
\end{align}
In particular, as $\theta \searrow -\frac{\pi}{2}$, 
\begin{align}
& \frac{d\nu(ai+ae^{i\theta})}{d\theta} = \frac{a^{2}(c-a)}{4c\pi}(\theta + \tfrac{\pi}{2})^{2}\big(1+\bigO(\theta + \tfrac{\pi}{2})\big), \label{asymp of dnu near 0 uniform 1} \\
& \frac{d\nu(ci+ce^{i\theta})}{d\theta} = \frac{c^{2}(c-a)}{4a\pi}(\theta + \tfrac{\pi}{2})^{2}\big(1+\bigO(\theta + \tfrac{\pi}{2})\big), \label{asymp of dnu near 0 uniform 2}
\end{align}
and, as $r\to 0$,
\begin{align}\label{asymp of nu near 0 uniform}
\nu\big(\partial \Omega \cap B_r(0) \cap \{z:\re z>0\}\big) = \frac{1}{2} \nu(\partial \Omega \cap B_r(0)) = \frac{c-a}{a c} \frac{r^{3}}{6\pi} \big( 1+\bigO(r) \big).
\end{align}
\end{theorem}

\begin{remark}\label{remark:check the main thm}
Since
\begin{align*}
\Theta_{R}(r) := \frac{1}{r} \big|\Omega_{a,c} \cap \partial B_{r}(0) \cap \{z:\re z>0\}\big| & 
= \arcsin\bigg(\frac{r}{2a}\bigg) - \arcsin\bigg(\frac{r}{2c}\bigg)
	\\
& = \frac{c-a}{2a c}r + \bigO(r^{3}) \qquad \mbox{as } r \to 0,
\end{align*}
the region $\Omega_{a,c}$ has two cusps with order of tangency $d = 1$ and coefficient of tangency $\frac{c-a}{2a c}$ at the origin.
An application of Theorem \ref{mainth2} with $b = 1$, $m_1 = 0$, and $m_2 = 2$ therefore implies that $\pi \nu = \mathrm{Bal}(d^{2}z,\partial \Omega_{a,c})$ satisfies
\begin{align}\label{nuOmegaac}
\pi  \nu(\partial \Omega \cap B_r(z_0)) = (1+o(1)) \frac{c-a}{3a c}r^{3}
\end{align}
as $r \to 0$, which is consistent with \eqref{asymp of nu near 0 uniform}.
\end{remark}

\subsection{Application to gap probabilities for Coulomb gases}

The Coulomb gas model for $n$ points on the plane with external potential $Q:\C\to \R\cup\{+\infty\}$ is the probability measure
\begin{align}
& \frac{1}{Z_{n}}\prod_{1\leq j<k \leq n}|z_{j}-z_{k}|^{\beta} \prod_{j=1}^{n} e^{-n \frac{\beta}{2} Q(z_{j})}d^{2}z_{j}, & & z_{1},\ldots,z_{n}\in \C, \label{general density intro}
\end{align}
where $Z_{n}$ is the normalization constant and $\beta >0$ is the inverse temperature. 

In this paper we will focus on the elliptic Ginibre point process, which corresponds to the external potential 
\begin{align*}
Q(z) = \frac{1}{1-\tau^{2}}\big( |z|^{2}-\tau \, \re z^{2} \big), \qquad \mbox{for some } \tau \in [0,1).
\end{align*}
This model has been widely studied, see e.g. \cite{BFreview} for a recent review, and is illustrated in Figure \ref{fig:density complement ellipse Ginibre} (left). As $n \to \infty$ and to a first order approximation, the points $z_{1},\ldots,z_{n}$ accumulate uniformly on the elliptical region $S$ defined by
\begin{align}\label{Sdef}
S=\Big\{ z\in \C: \Big( \frac{\re z}{1+\tau} \Big)^{2} + \Big( \frac{\im z}{1-\tau} \Big)^{2} \leq 1 \Big\}.
\end{align}
Let $\mu$ be the uniform measure on $S$ given by
\begin{align}\label{mu S EG beginning of intro}
d\mu(z) = \begin{cases} \frac{d^{2}z}{\pi(1-\tau^{2})} & \text{if $z \in S$},
	\\
0 & \text{if $z \notin S$}.
\end{cases}
\end{align}
\textit{Large gap asymptotics} for Coulomb gases are problems consisting in obtaining explicit asymptotic formulas as $n\to + \infty$ for the hole probability $\mathbb{P}(\#\{z_{j}\in \Omega\}=0)$, where $\Omega$ is a given subset of the plane. If $\partial \Omega \subset S$, the leading order term is given in terms of the balayage measure $\mathrm{Bal}(\hat{\mu}|_{\Omega},\partial \Omega)$, where $\hat{\mu}$ is the equilibrium measure associated with the Coulomb gas (for the elliptic Ginibre point process, $\hat{\mu}=\mu$ is given by \eqref{mu S EG beginning of intro}) \cite{AR2017, C2023}.
Large gap asymptotics have been widely studied; see e.g. \cite{A2018, AR2017, APS2009, BP2024, C2021, C2023, ForresterHoleProba, JLM1993} for results in the case where $\Omega\subset S$ has a smooth boundary, and \cite{AR2017, C2023} for some cases where $\Omega\subset S$ has a corner (see also \cite{LMS2018} for important applications to physics). Theorem \ref{thm:EG tacnode C} below provides the first large gap result in a situation where $\Omega$ has a cusp. 

Let $z_0 + e^{i\theta_0}\Omega_{a,c}$ denote the tacnodal region $\Omega_{a,c}$ defined in (\ref{Omegaacdef}) rotated by an angle $\theta_0 \in (-\pi,\pi]$ and then translated by $z_{0}\in \C$:
\begin{align}\label{z0eitheta0Omegaac}
z_0 + e^{i\theta_0}\Omega_{a,c} = \{z:|z-(z_{0}+cie^{i\theta_{0}})|<c\} \setminus \{z:|z-(z_{0}+aie^{i\theta_{0}})|\leq a\}.
\end{align}

Let $\psi(z) := \frac{\Gamma'(z)}{\Gamma(z)}$ be the digamma function, where $\Gamma(z) := \int_{0}^{+\infty}t^{z-1}e^{-t}dt$ is the Gamma function. 
The proof of the following theorem is presented in Section \ref{holesec}.

\begin{figure}
\begin{center}
\includegraphics[width=10cm]{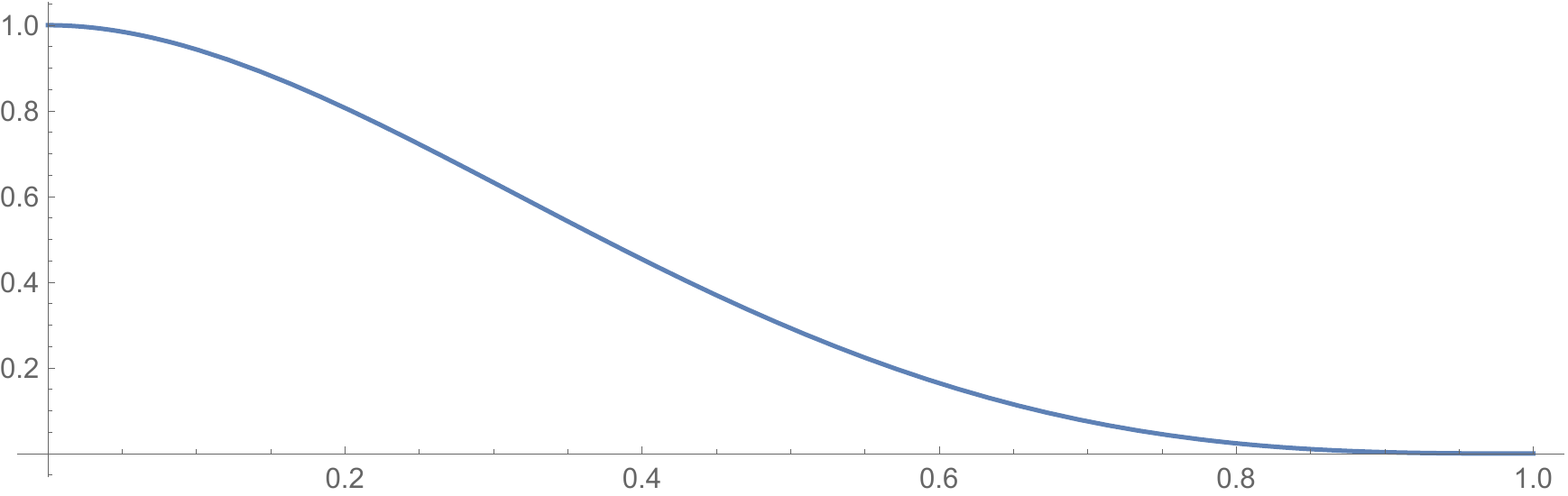}
\end{center}
\caption{\label{fig:F}The function $x \mapsto F(x)$ defined in \eqref{def of F}.}
\end{figure}

\begin{theorem}[Gap probability]\label{thm:EG tacnode C}
Let $\beta>0$. Let $\tau\in [0,1)$, $z_{0}\in \C$, $\theta_{0}\in (-\pi,\pi]$, and $0<a<c$ be such that the tacnodal region $z_0 + e^{i\theta_0}\Omega_{a,c}$ in (\ref{z0eitheta0Omegaac}) is a subset of the elliptical region $S$ defined in (\ref{Sdef}). 
As $n \to +\infty$, we have 
\begin{align}\label{asymp hole in thm}
\mathbb{P}(\# \{z_{j}\in z_0 + e^{i\theta_0}\Omega_{a,c}\} = 0) = \exp \big( -C n^{2}+o(n^{2}) \big)
\end{align}
with $C=\frac{\beta c^{4}}{8(1-\tau^{2})^{2}}F(\frac{a}{c})$, where $F(x)$ is defined for $x\in (0,1)$ by
\begin{align}\label{def of F}
F(x) = 1 -x^4-4 x^2 \psi'\left(\frac{1}{1-x}\right).
\end{align}
Using well-known special values of $\psi'$ (see e.g. \cite[5.4.14, 5.4.15, 5.15.4]{NIST}), we get
\begin{align*}
& C|_{\frac{a}{c}=\frac{1}{2}}=\frac{93-8\pi^{2}}{384} \beta c^{4}, & & C|_{\frac{a}{c}=\frac{1}{3}}=\frac{112-9\pi^{2}}{324} \beta c^{4}, & & C|_{\frac{a}{c}=\frac{2}{3}}=\frac{245-24\pi^{2}}{648}\beta c^{4}, \\
& C|_{\frac{a}{c}=\frac{3}{4}}=\frac{959-96\pi^{2}}{2048} \beta c^{4}, & & C|_{\frac{a}{c}=\frac{3}{5}}=\frac{2272-225\pi^{2}}{2500} \beta c^{4}, & & C|_{\frac{a}{c}=\frac{4}{5}}=\frac{23821-2400\pi^{2}}{45000}\beta c^{4}.
\end{align*}
\end{theorem}

As $\frac{a}{c}$ increases from $0$ to $1$ with $c$ fixed, the tacnodal region $z_0 + e^{i\theta_0}\Omega_{a,c}$ shrinks to the empty set, and hence the probability that there are no points in $z_0 + e^{i\theta_0}\Omega_{a,c}$ increases to $1$. This means that $F(x)$ should be a decreasing function of $x \in (0,1)$ that approaches $0$ as $x \to 1$. On the other hand, when $a=0$, $z_0 + e^{i\theta_0}\Omega_{a,c}$ is a disk, in which case it is known from \cite[Remark 2.11 and Theorem 2.24]{C2023} that \eqref{asymp hole in thm} holds with $C= \frac{\beta c^{4}}{8(1-\tau^{2})^{2}}$. Hence we also expect $F(0)=1$. We verify these properties in the next proposition, whose proof is included in Appendix \ref{Fapp}; see also Figure \ref{fig:F}. 

\begin{proposition}\label{Fprop}
The function $F(x)$ defined in (\ref{def of F}) for $x \in (0,1)$ extends to a continuous strictly decreasing function of $x \in [0,1]$ with $F(0) = 1$ and $F(1) = 0$. The behavior of $F(x)$ near $x = 0$ and $x = 1$ is
\begin{align}\label{Fnear0and1}
& F(x) = 1-\frac{2\pi^{2}}{3}x^{2} + \bigO(x^{3}) & & \mbox{as } x \to 0, \\
& F(x) = \frac{10}{3}(1-x)^{3} + \bigO((x-1)^{4}) & & \mbox{as } x \to 1.
\end{align}
\end{proposition}

\section{A single cusp: proof of Theorem \ref{mainth}}\label{proofsec}
Let $\Omega$ be a finitely connected Jordan domain in $\C^*$ with a cusp with order of tangency $d > 0$ and coefficient of tangency $a>0$ at a point $z_0 \in \partial \Omega \cap \C$. By applying a translation, we may assume without loss of generality that $z_0 = 0$. Fix $\epsilon>0$ and $b > 0$. By choosing $\rho_0 > 0$ small enough, we may assume that $\Omega \cap \partial B_r(0)$ is a connected set whose length $r\Theta(r)$ obeys
\begin{align}\label{Thetasingle}
r\Theta(r) \leq (1+\epsilon)a r^{d+1} \qquad \text{for $0 < r \leq \rho_0$}
\end{align}
and that (\ref{omegaestimate}) holds for a single cusp at $0$, so that if $r > 0$ and $z \in \Omega$, then 
\begin{align}\label{omegaleq8pirzsingle}
\omega(z, \partial \Omega \cap B_{r}(0), \Omega) \leq 
\begin{cases}
\frac{8}{\pi} \exp\big( -\pi \frac{r^{-d}-|z|^{-d}}{(1+\epsilon)a d} \big)   & \text{if $0 < r < |z| \leq \rho_0$},
	\\
\frac{8}{\pi} \exp \big( -\pi \frac{r^{-d}-\rho_{0}^{-d}}{(1+\epsilon)a d} \big) & \text{if $0 < r < \rho_0 \leq |z|$}.
\end{cases}
\end{align}

Let $\mu$ be a non-negative measure on $\Omega$ of finite total mass such that $d\mu(z) = (1+o(1)) |z|^{2b-2}d^{2}z$ as $z\to 0$. 
Shrinking $\rho_0>0$ if necessary, we may assume that
\begin{align}\label{c1muc2}
(1-\epsilon) \int_{A}|z|^{2b-2}d^{2}z \leq \mu(A) \leq (1+\epsilon) \int_{A}|z|^{2b-2}d^{2}z
\end{align} 
for all measurable subsets $A$ of $\Omega \cap \overline{B_{\rho_0}(0)}$. Let $\nu = \mathrm{Bal}(\mu,\partial \Omega)$.

\subsection{Proof of upper bound}
We first prove the upper bound in \eqref{nuestimate eps}.

Recall that, by definition, $\nu(\partial \Omega \cap B_r(0)) = \int_\Omega \omega(z, \partial \Omega \cap B_r(0), \Omega) d\mu(z)$. For $r \in (0, \rho_0)$ we write $\nu(\partial \Omega \cap B_r(0))$ as the sum of three integrals:
$$\nu(\partial \Omega \cap B_r(0)) = I_{1,r} + I_{2,r} + I_{3,r},$$
where
\begin{align*}
&  I_{1,r} := \int_{\Omega \cap \overline{B_{r}(0)}} \omega(z, \partial \Omega \cap B_r(0), \Omega) d\mu(z), 
&& I_{2,r} := \int_{(\Omega \cap B_{\rho_0}(0))\setminus \overline{B_{r}(0)}} \omega(z, \partial \Omega \cap B_r(0), \Omega) d\mu(z),
	\\
&
I_{3,r} := \int_{\Omega \setminus B_{\rho_0}(0)} \omega(z, \partial \Omega \cap B_r(0), \Omega) d\mu(z).
\end{align*}
Using the fact that the harmonic measure of any set is $\leq 1$, (\ref{c1muc2}), and (\ref{Thetasingle}), we find
\begin{align}
I_{1,r} & \leq \int_{\Omega \cap \overline{B_{r}(0)}}  d\mu(z)
\leq (1+\epsilon) \int_{\Omega \cap \overline{B_{r}(0)}}  |z|^{2b-2}d^{2}z
\leq  (1+\epsilon)^{2} a \int_0^r \rho^{2b + d-1} d\rho \nonumber \\
& = \frac{(1+\epsilon)^2 a}{2b+d}r^{2b+d} \label{I1restimate}
\end{align}
for all $r \in (0, \rho_0)$, provided that $\rho_{0}>0$ is chosen small enough. To estimate $I_{2,r}$, we use (\ref{Thetasingle}), (\ref{omegaleq8pirzsingle}), and (\ref{c1muc2}) to write, for $r \in (0, \rho_0)$ (shrinking again $\rho_{0}$ if necessary),
\begin{align}\nonumber
I_{2,r} 
	& \leq  (1+\epsilon)\int_{(\Omega \cap B_{\rho_0}(0))\setminus \overline{B_{r}(0)}} \frac{8}{\pi} \exp \bigg( -\pi \frac{r^{-d}-|z|^{-d}}{(1+\epsilon)  a d} \bigg) |z|^{2b-2}d^{2}z \nonumber
	\\
& \leq (1+\epsilon) \int_r^{\rho_0} \frac{8}{\pi} \exp \bigg( -\pi \frac{r^{-d} - \rho^{-d}}{(1+\epsilon) a d} \bigg) a \rho^{d}(1+ \epsilon) \rho^{2b-1}d\rho \nonumber \\
& \leq (1+\epsilon)^2\tilde{I}_{2,r}, \label{I2restimate}
\end{align}
where
\begin{align*}
\tilde{I}_{2,r} := \int_r^{\rho_0} \frac{8 a}{\pi} \exp \bigg( -\pi \frac{r^{-d} - \rho^{-d}}{(1+\epsilon) a d} \bigg)  \rho^{2b+d-1}d\rho.
\end{align*}
Let $R=\frac{\rho_{0}}{r}$ and $a_{2}:=(1+\epsilon) a$. In what follows, we use $C>0$ and $c>0$ to denote generic strictly positive constants. With the changes of variables $\rho = r\tilde{\rho}$ and $y=r^{-d}(1-\tilde{\rho}^{-d})$, we get
\begin{align*}
\tilde{I}_{2,r} & = \frac{8a}{\pi} r^{2b+d} \int_{1}^{R} \exp \bigg( -\pi r^{-d} \frac{1-\tilde{\rho}^{-d}}{a_{2}d} \bigg) \tilde{\rho}^{2b+d-1}d\tilde{\rho} \\
& = \frac{8a}{\pi d}r^{2b+2d} \int_{0}^{r^{-d}(1-R^{-d})} \frac{\exp ( -\frac{\pi y}{a_{2}d} )}{(1-r^{d}y)^{\frac{2b}{d}+2}}dy \\
& = \frac{8a}{\pi d}r^{2b+2d} \bigg( \int_{0}^{\frac{r^{-d}}{2}} \frac{\exp ( -\frac{\pi y}{a_{2}d} )}{(1-r^{d}y)^{\frac{2b}{d}+2}}dy + \int_{\frac{r^{-d}}{2}}^{r^{-d}(1-R^{-d})} \frac{\exp ( -\frac{\pi y}{a_{2}d} )}{(1-r^{d}y)^{\frac{2b}{d}+2}}dy \bigg) 
	\\
& \leq \frac{8a}{\pi d}r^{2b+2d} \bigg( \int_{0}^{\frac{r^{-d}}{2}} \frac{\exp ( -\frac{\pi y}{a_{2}d} )}{(1/2)^{\frac{2b}{d}+2}} dy 
+ \int_{\frac{r^{-d}}{2}}^{r^{-d}(1-R^{-d})} \frac{\exp(-\frac{\pi}{a_{2}d}\frac{r^{-d}}{2})}{R^{-2b-2d}}dy
 \bigg) 
	\\
& \leq C r^{2b+2d} \bigg( 1 
+ r^{-2b-3d} e^{-\frac{\pi}{a_{2}d}\frac{r^{-d}}{2}} 
 \bigg) 
  \leq C r^{2b+2d}.
\end{align*}
for all $r \in (0, \rho_0)$.
Finally, using (\ref{omegaleq8pirzsingle}) and the fact that $\mu$ has finite total mass, we obtain
\begin{align}\label{I3restimate}
I_{3,r} 
\leq  \frac{8}{\pi} \exp \bigg( -\pi\frac{r^{-d}-\rho_{0}^{-d}}{a_{2} d} \bigg) \int_{\Omega \setminus B_{\rho_0}(0)} d\mu(z)
\leq C e^{-c r^{-d}}
\end{align}
for $r \in (0, \rho_0)$. Combining $\nu(\partial \Omega \cap B_r(0)) = I_{1,r} + I_{2,r} + I_{3,r}$ with (\ref{I1restimate}), (\ref{I2restimate}), and (\ref{I3restimate}), we find after shrinking $\rho_{0}$ if necessary that
$$\nu(\partial \Omega \cap B_r(0)) \leq \frac{(1+\epsilon)^3 a}{2b+d} r^{2b+d} \qquad \text{for all $r \in (0, \rho_0)$},
$$
which is the desired upper bound on $\nu(\partial \Omega \cap B_r(0))$ stated in \eqref{nuestimate eps} (since $\epsilon >0$ is arbitrary).

\subsection{Proof of lower bound}
The open set $U := \Omega \cap B_{\rho_0}(0)$ has a cusp at $0$. By applying a rotation if necessary, we may assume that the cusp is tangent to $\R_{\geq 0}$. 
Furthermore, shrinking $\rho_0 > 0$ if necessary, we may assume that $U$ is simply connected and that there exists a $p>0$ such that the map $\varphi(z) = z^p$ is bijective on $U$ and $\varphi(U)$ has an analytic cusp at $0$ with small perturbation of angles. 

Using that $\omega(z, \partial U \cap B_r(0), U)$ is everywhere nonnegative, we get
$$\nu(\partial \Omega \cap B_r(0)) 
= \int_\Omega \omega(z, \partial U \cap B_r(0), \Omega) d\mu(z)
\geq \int_{U} \omega(z, \partial U \cap B_r(0), \Omega) d\mu(z).$$
For $r \in (0, \rho_0]$, we have $\partial \Omega \cap B_r(0) = \partial U \cap B_r(0)$ and, by the maximum principle, $\omega(z, \partial U \cap B_r(0), \Omega) \geq \omega(z, \partial U \cap B_r(0), U)$ for all $z\in U$. Hence, using also (\ref{c1muc2}), we obtain
$$\nu(\partial \Omega \cap B_r(0)) \geq \int_U \omega(z, \partial U \cap B_r(0), U) d\mu(z) \geq (1-\epsilon) \int_U \omega(z, \partial U \cap B_r(0), U) |z|^{2b-2}d^{2}z.$$
To complete the proof of Theorem \ref{mainth} it is therefore enough to show that
\begin{align}\label{intUomega}
\int_U \omega(z, \partial U \cap B_r(0), U) |z|^{2b-2}d^{2}z \geq c_{0} r^{2b+d}
\end{align}
for all sufficiently small $r > 0$, where $c_{0}=(1-\epsilon)\frac{a}{2b+d}$. 

The next lemma shows that (\ref{intUomega}) will follow if we can show a related bound where $U$ is replaced by $\varphi(U)$.

\begin{lemma}\label{varphilemma}
If
\begin{align}\label{intvarphiU}
I_r := \int_{\varphi(U)} \omega(z, \partial \varphi(U) \cap B_r(0), \varphi(U)) |z|^{2\frac{b}{p}-2}d^{2}z \geq c_{0} p^{2} r^{\frac{2b+d}{p}}
\end{align}
for all sufficiently small $r > 0$, then (\ref{intUomega}) holds for all sufficiently small $r > 0$. 
\end{lemma}
\begin{proof}
Making the change of variables $z = \varphi(v) = v^p$ and using that $\varphi:U \to \varphi(U)$ is a bijection, we obtain
\begin{align*}
c_{0} p^{2} r^{\frac{2b+d}{p}} 
\leq I_r & = \int_U \omega(\varphi(v), \partial \varphi(U) \cap B_r(0), \varphi(U)) 
|\varphi(v)|^{2\frac{b}{p}-2} p^2 |v|^{2p-2} d^2 v 
	\\
& = 
\int_{U} \omega(v, \partial U \cap B_{r^{1/p}}(0), U) p^2 |v|^{2b-2} d^{2}v, 
\end{align*}
where the second equality follows by the conformal invariance of harmonic measure.
Replacing $r$ by $r^p$, the desired conclusion follows. 
\end{proof}

We will prove (\ref{intvarphiU}) by applying the following result. Let $\mathbb{H} := \{z\in \C: \im z >0\}$.

\begin{lemma}\label{Kaiserlemma}
  Let $\Omega_1 \subset \C$ be a simply connected domain with an analytic cusp at $0$ with small perturbation of angles which is tangent to $\R_{\geq 0}$ at $0$. Let $f:\mathbb{H} \to \Omega_1$ be a conformal map with $f(0) = 0$. Then, as $w \in \mathbb{H}$ tends to $0$,
\begin{align}
& f(w) = \bigg(-\frac{\pi}{d_1 a_1 \log |w|}\bigg)^{\frac{1}{d_1}}(1 + o(1))  \label{fnear0} \\
& f'(w) = f(w)\frac{-1}{d_{1}w\log |w|}(1+o(1)). \label{fdernear0}
\end{align}
where $d_1$ is the order of tangency and $a_1$ is the coefficient of tangency of the analytic cusp. 
\end{lemma}
\begin{proof}
The first identity \eqref{fnear0} is a consequence of \cite[Theorem 10]{K2010}. The second identity \eqref{fdernear0} can be proved as follows. By \cite[Theorem 10 and proofs of Theorems 9 and 11]{K2010}, the functions $f:\mathbb{H}\to \Omega_{1}$ and $f^{-1}:\Omega_{1}\to \mathbb{H}$ satisfy
\begin{align}
& f(w) = g(w)\xi(w), & & g(w):=\bigg(-\frac{\pi}{d_1 a_1 \log w}\bigg)^{\frac{1}{d_1}}, & & f^{-1}(z) = \exp \bigg( \frac{-\pi}{d_{1}a_{1}z^{d_{1}}} \bigg)\zeta(z), \label{lol24}
\end{align}
where the principal branch is used for the log and the roots, $\xi:\mathbb{H}\to \mathbb{C}$ is analytic and satisfies $\xi(w) = 1+o(1)$ as $w\to 0$ in $\mathbb{H}$, and $\zeta:\Omega_{1}\to \mathbb{C}$ is analytic and satisfies $\zeta(z)=\ell + o(1)$ as $z\to 0$ in $\Omega_{1}$ for some $\ell \in \mathbb{C}\setminus\{0\}$. Moreover, $\xi$ can be analytically extended to a slightly bigger domain, so that
\begin{align}\label{lol23}
\xi'(w) = \frac{1}{2\pi i}\int_{\partial B_{|w|/2}(w)}\frac{\xi(w')}{(w'-w)^{2}}dw'
\end{align}
holds for all $w\in \mathbb{H}$ that are sufficiently close to $0$. A more precise asymptotic formula for $\xi(w)$ as $w\to 0$ in $\mathbb{H}$ can be obtained as follows: as $z\to 0$ in $\Omega_{1}$,
\begin{align*}
f(f^{-1}(z)) & = z = \bigg(-\frac{\pi}{d_1 a_1 \log \Big[ \exp \big( \frac{-\pi}{d_{1}a_{1}z^{d_{1}}} \big) (\ell + o(1))\Big]}\bigg)^{\frac{1}{d_1}} \xi(f^{-1}(z)) \\
& = z \bigg(\frac{1}{1-\frac{d_{1}a_{1}z^{d_{1}}}{\pi}\log \ell + o(z^{d_{1}})}\bigg)^{\frac{1}{d_1}} \xi(f^{-1}(z)) = z \bigg(1+\frac{a_{1}\log \ell}{\pi}z^{d_{1}}+o(z^{d_{1}})\bigg)\xi(f^{-1}(z)).
\end{align*}
It follows that $\xi(f^{-1}(z)) = 1-\frac{a_{1}\log \ell}{\pi}z^{d_{1}}+o(z^{d_{1}})$ as $z\to 0$ in $\Omega_{1}$, or equivalently,
\begin{align*}
\xi(w) = 1-\frac{a_{1}\log \ell}{\pi}f(w)^{d_{1}}+o(f(w)^{d_{1}}) = 1+\frac{\log \ell}{d_1 \log w}+o\Big(\frac{1}{\log w}\Big) \qquad \mbox{as } \mathbb{H}\ni w \to 0.
\end{align*}
One can also rewrite the above as
\begin{align*}
\xi(w) = 1+\frac{\log \ell}{d_1 \log_{-\pi/2}(w)}+o\Big(\frac{1}{\log w}\Big) \qquad \mbox{as } \mathbb{H}\ni w \to 0,
\end{align*}
where $\log_{-\pi/2}(w) := \log |w| + i \arg_{-\pi/2}(w)$ and $\arg_{-\pi/2}(w) \in (-\frac{\pi}{2},\frac{3\pi}{2})$. Since $1+\frac{\log \ell}{d_1 \log_{-\pi/2}(w)}$ is analytic in $B_{|w|/2}(w)$ for all $w\in \mathbb{H}$, by combining the above with \eqref{lol23}, we infer that $\xi'(w) = o(w^{-1}(\log w)^{-1})$ as $\mathbb{H}\ni w\to 0$. 
Since $g'(w) = g(w) \frac{-1}{d_{1}w\log w}$, by \eqref{lol24} we have
\begin{align*}
f'(w) = g'(w)\xi(w) + g(w)\xi'(w) = f(w)\frac{-1}{d_{1}w\log w}(1+o(1)) \qquad \mbox{as } \mathbb{H}\ni w \to 0,
\end{align*}
from which \eqref{fdernear0} follows.
\end{proof}

In view of Lemma \ref{varphilemma}, the next lemma completes the proof of Theorem \ref{mainth}.

\begin{lemma}
For any fixed $\epsilon >0$, the estimate (\ref{intvarphiU}) holds for all sufficiently small $r > 0$. 
\end{lemma}
\begin{proof}
Let $a > 0$ be such that $\Theta(r) = ar^d(1 + \bigO(r^{d+p}))$ as $r \to 0$. Let $r\Theta_1(r)$ be the length of $\varphi(U) \cap \partial B_r(0)$. Then
$\Theta_1(r) = p \Theta(r^{\frac{1}{p}}) = pa r^{\frac{d}{p}} (1 + \bigO(r^{\frac{d}{p}+1}))$ as $r \to 0$, so the cusp of $\varphi(U)$ at $0$ has order of tangency $d_1 := d/p$ and coefficient of tangency $a_1 := pa$. 

Let $f:\mathbb{H} \to \varphi(U)$ be a conformal map with $f(0) = 0$. By Lemma \ref{Kaiserlemma}, $f$ satisfies (\ref{fnear0}) and \eqref{fdernear0}. It follows that $f$ obeys the following two estimates: for any fixed $\epsilon >0$,
\begin{align}\label{ffprimebounds}
|f(w)| \geq (1-\epsilon)\bigg(\frac{\frac{\pi}{d_{1}a_{1}}}{|\log|w||}\bigg)^{\frac{1}{d_{1}}} 
\quad \text{and} \quad |f'(w)| \geq  \frac{(\frac{\pi}{d_{1}a_{1}})^{\frac{1}{d_{1}}}}{d_{1}|w|} \frac{1-\epsilon}{|\log|w||^{1 + \frac{1}{d_1}}} 
\end{align}
for all small enough $w$. 

Performing the change of variables $z = f(w)$, the integral $I_r$ defined in (\ref{intvarphiU}) can be written as
$$I_r = \int_{\mathbb{H}} \omega(f(w), \partial \varphi(U) \cap B_r(0), \varphi(U)) |f(w)|^{2\frac{b}{p}-2} |f'(w)|^2 d^{2}w.
$$
Using the conformal invariance of harmonic measure, the estimates (\ref{ffprimebounds}), and the relations $d_1 p = d$ and $a_{1}=pa$, we deduce that
\begin{align}\label{Irgeq}
I_r & \geq  (1-\epsilon)^{\frac{2b}{p}}\bigg(  \frac{\pi}{d a} \bigg)^{\frac{2b}{p}} \frac{p^{2}}{d^{2}} \int_{\mathbb{H}} \omega(w, f^{-1}(\partial \varphi(U) \cap B_r(0)), \mathbb{H}) \frac{1}{|\log|w||^{\frac{2b}{d} + 2}}
 \frac{1}{|w|^2}
 d^{2}w.
\end{align}

Solving (\ref{fnear0}) for $w = f^{-1}(z)$, we find
$$f^{-1}(z) = e^{-\frac{\pi}{a_1 d_1 z^{d_1}} (1+o(1))}
= e^{-\frac{\pi}{a d z^{d/p}} (1+o(1))} \qquad \text{as $z \in \varphi(\Omega)$ tends to $0$},$$
and so
\begin{align}\label{finverseinterval}
f^{-1}(\partial \varphi(U) \cap B_r(0)) \supset \bigg[-e^{-\frac{(1+\epsilon)\pi}{a d r^{d/p}}}, e^{-\frac{(1+\epsilon)\pi}{a d r^{d/p}}}\bigg]
\end{align}
for all sufficiently small $r > 0$.
If $[a_{0},b_{0}]$ is an interval, then $\omega(w, [a_{0},b_{0}], \mathbb{H})$ is the angle (divided by $\pi$) of the interval $[a_{0},b_{0}]$ as seen from $w$, i.e.,
$$\omega(w, [a_{0},b_{0}], \mathbb{H})
= \frac{1}{\pi}\bigg(\arctan\bigg(\frac{\re(w) - a_{0}}{\im w}\bigg) - \arctan\bigg(\frac{\re(w) - b_{0}}{\im w}\bigg)\bigg).
$$
Let $w = \rho e^{i\theta}$. If $[-u,u]$ is a symmetric interval around $0$, then for all $\theta \in (0, \pi)$ and for all $\rho \in (0,u)$, we have
\begin{align}\label{omegawuuH}
\omega(w, [-u,u], \mathbb{H})
\geq \omega(i \rho, [-u,u], \mathbb{H})
= \frac{2}{\pi} \arctan\bigg(\frac{u}{\rho}\bigg).
\end{align}
For $\rho \in (0, e^{-\frac{(1+2\epsilon)\pi}{a d r^{d/p}}})$ and all sufficiently small $r>0$, we have 
\begin{align*}
\frac{2}{\pi}\arctan\bigg(\frac{e^{-\frac{(1+\epsilon)\pi}{a d r^{d/p}}}}{\rho}\bigg) \geq 1-\epsilon.
\end{align*}
We conclude from (\ref{Irgeq}), (\ref{finverseinterval}), and (\ref{omegawuuH}) with $u = e^{-\frac{(1+\epsilon)\pi}{a d r^{d/p}}}$ that
\begin{align*}
I_r & \geq (1-\epsilon)^{\frac{2b}{p}+1}\bigg(  \frac{\pi}{d a} \bigg)^{\frac{2b}{p}} \frac{p^{2}}{d^{2}} \int_{0}^{\pi} \int_0^{e^{-\frac{(1+2\epsilon)\pi}{a d r^{d/p}}}} \frac{1}{|\log \rho|^{\frac{2b}{d} + 2} \rho^2}  \rho d\rho d\theta \\
& = (1-\epsilon)^{\frac{2b}{p}+1}\bigg(  \frac{\pi}{d a} \bigg)^{\frac{2b}{p}} \frac{p^{2}}{d^{2}} \pi  \int_0^{e^{-\frac{(1+2\epsilon)\pi}{a d r^{d/p}}}}  \frac{1}{|\log \rho|^{\frac{2b}{d} + 2} \rho}   d\rho.
\end{align*}
Since
\begin{align*}
\int_0^{e^{-\frac{(1+2\epsilon)\pi}{a d r^{d/p}}}}  \frac{(\log \frac{1}{\rho})^{-\frac{2b}{d} -2} }{\rho} d\rho
= \frac{(\log\frac{1}{\rho})^{-\frac{2b}{d}-1}}{1+\frac{2b}{d}}\bigg|_0^{e^{-\frac{(1+2\epsilon)\pi}{ a d r^{d/p}}}}
= \frac{(\frac{(1+2\epsilon)\pi}{a d r^{d/p}})^{-\frac{2b}{d}-1}}{1+\frac{2b}{d}},
\end{align*}
we obtain
\begin{align*}
I_{r} \geq \frac{(1-\epsilon)^{\frac{2b}{p}+1}}{(1+2\epsilon)^{\frac{2b}{d}+1}}p^{2} \frac{a}{2b+d} r^{\frac{2b+d}{p}},
\end{align*}
which is the desired estimate (since $\epsilon >0$ is arbitrary).
\end{proof}

\section{Multiple corners and cusps: proof of Theorem \ref{mainth2}}\label{proofsec2}
Suppose $\Omega$ is an open subset of $\C^*$ satisfying $(i)$ and $(ii)$ of Theorem \ref{mainth2} with $z_0 = 0$. Shrinking $\rho_0 > 0$ if necessary, we may assume that the conclusions of Lemma \ref{omegaupperboundlemma} hold.
Let $\alpha := \max_{1\leq j \leq m_1} \alpha_j$ and $d := \min_{m_1 +1\leq j \leq m} d_j$. Then $\pi \alpha$ is the largest opening of the corners at $0$ and $d$ is the smallest order of tangency of the cusps at $0$. 

Suppose $\mu$ is a non-negative measure of finite total mass on $\Omega$ satisfying $d\mu(z) =(1+o(1)) |z-z_{0}|^{2b-2}d^{2}z$ as $z\to 0$. 
Let $\nu := \mathrm{Bal}(\mu,\partial \Omega)$ and let $\nu_j := \mathrm{Bal}(\mu|_{U_j},\partial U_j)$ be the balayage of the restriction of $\mu$ to the component $U_j$ of $\Omega \cap B_{\rho_0}(z_0)$.

The leading behavior of $\nu(\partial \Omega \cap B_r(0))$ as $r \to 0$ receives contributions from the corners and cusps at $0$. The next lemma shows that the contributions of the corners and cusps decouple and can be computed locally up to terms of order $\bigO(r^{1/\alpha})$. Furthermore, if there are no corners at $0$, the cusps decouple up to exponentially small terms.
The reason behind the decoupling is that it is unlikely that a Brownian motion starting from a point in $U_j$ first leaves $U_j$, then enters another component $U_i$ with $i \neq j$, and then makes it all the way to $\partial U_i \cap B_r(0)$ without leaving $\Omega$. 

In this section, we use $C>0$ and $c>0$ to denote generic strictly positive constants.

\begin{lemma}[Decoupling and localization]\label{decouplinglemma}
If $\Omega$ has at least one corner at $0$ (i.e., if $m_1 \geq 1$), then 
\begin{align}\label{decouplingestimatecorner}
\nu(\partial \Omega \cap B_r(0)) - C r^{\frac{1}{\alpha}} \leq \sum_{j=1}^m \nu_j(\partial U_j \cap B_r(0)) \leq \nu(\partial \Omega \cap B_r(0))
\end{align}
for all sufficiently small $r > 0$.
If $\Omega$ has only cusps at $0$ (i.e., if $m_1 =0$), then 
\begin{align}\label{decouplingestimatecusp}
\nu(\partial \Omega \cap B_r(0)) - C e^{-c r^{-d}} \leq \sum_{j=1}^m \nu_j(\partial U_j \cap B_r(0)) \leq \nu(\partial \Omega \cap B_r(0))
\end{align}
for all sufficiently small $r > 0$.
\end{lemma}
\begin{proof}
In the absence of cusps, the estimate (\ref{decouplingestimatecorner}) was proved in \cite[Lemma 6.1]{CL Corner}. The proof relies on the estimate (\ref{omegaleq8pirzmultiple2}) of Lemma \ref{omegaupperboundlemma} and the same arguments yield (\ref{decouplingestimatecorner}) also in the presence of a finite number of cusps. 

The proof of (\ref{decouplingestimatecusp}) follows the same steps, except that if $U_j$ has a cusp, then (\ref{omegaleq8pirzmultiple2}) implies that $\omega(z, \partial U_j \cap B_{r}(0), \Omega) \leq C e^{-c r^{-d_j}}$ for all $0<r<\rho_{0}\leq |z|$, $z\in \Omega$. (If $U_j$ has a corner, then  it instead holds that $\omega(z, \partial U_j \cap B_{r}(0), \Omega) \leq C r^{1/\alpha_j}$ for all $0<r<\rho_{0}\leq |z|$, $z\in \Omega$.) Since $e^{-c r^{-d_j}} \leq e^{-c r^{-d}}$ for all $j = m_1 + 1, \dots, m$, the desired conclusion follows.  
\end{proof}

If $U_j$ has a cusp at $0$, then we can estimate $\nu_j(\partial U_j \cap B_r(0))$ by applying Theorem \ref{mainth} to the domain $U_j$. This yields
\begin{align}\label{nujcusp}
(1-\epsilon)\frac{a_{j}}{2b+d_{j}} r^{2b+d_{j}} \leq \nu_j(\partial U_j \cap B_r(0)) \leq (1+\epsilon)\frac{a_{j}}{2b+d_{j}} r^{2b+d_{j}} \qquad \text{for $j = m_1 + 1, \dots, m$}
\end{align}
and all sufficiently small $r>0$. On the other hand, if $U_j$ has a corner at $0$, then \cite[Theorem 2.1]{CL Corner} implies that
\begin{align}
\nonumber
(1-\epsilon) \frac{\tan(\pi \alpha_j b)}{2b^2} r^{2b} & \leq \nu_j(\partial U_j \cap B_r(0)) \leq (1+\epsilon) \frac{\pi \alpha_j}{2b} \bigg(1  + \frac{16 b}{\pi (\frac{1}{\alpha_j} - 2b)}\bigg)  r^{2b} && \text{if } 2b < \frac{1}{\alpha_{j}}, \nonumber 
	\\ \nonumber 
(1-\epsilon) \frac{2}{\pi b} r^{2b} \log(\tfrac{1}{r})
& \leq
 \nu_j(\partial U_j \cap B_r(0))
\leq (1+\epsilon) \frac{4}{b} r^{2b} \log(\tfrac{1}{r}) & & \text{if } 2b= \frac{1}{\alpha_{j}}, 	\\ \label{nujcorner}
c r^{\frac{1}{\alpha_{j}}} & \leq
 \nu_j(\partial U_j \cap B_r(0))
\leq C r^{\frac{1}{\alpha_{j}}} & & \hspace{-2.2cm} \text{if } 2b > \frac{1}{\alpha_{j}}.
\end{align}

Summing (\ref{nujcorner}) over $j$ from $1$ to $m_1$, and (\ref{nujcusp}) over $j$ from $m_1 + 1$ to $m$, we obtain, for all sufficiently small $r>0$,
\begin{align*}
(1-\epsilon)\frac{\sum_{j \in \mathcal{D}}a_{j}}{2b+d} r^{2b+d} & \leq \sum_{j=1}^m \nu_j(\partial U_j \cap B_r(0)) \leq (1+\epsilon)\frac{\sum_{j \in \mathcal{D}}a_{j}}{2b+d} r^{2b+d} & & \hspace{-1cm}  \mbox{if } m_{1}=0, \nonumber \\
\nonumber
(1-\epsilon) \sum_{j=1}^{m_{1}} \frac{\tan(\pi \alpha_j b)}{2b^2} r^{2b} & \leq \sum_{j=1}^m \nu_j(\partial U_j \cap B_r(0)) \leq (1+\epsilon) \sum_{j=1}^{m_{1}} \frac{\pi \alpha_j}{2b} \bigg(1  + \frac{16 b}{\pi (\frac{1}{\alpha_j} - 2b)}\bigg)  r^{2b} \\
&  && \hspace{-2.2cm} \text{if $m_{1}\geq 1$ and $2b < \frac{1}{\alpha}$}, \nonumber 
	\\ \nonumber 
m_{\alpha} (1-\epsilon) \frac{2}{\pi b} r^{2b} \log(\tfrac{1}{r})
& \leq
 \sum_{j=1}^m \nu_j(\partial U_j \cap B_r(0))
\leq m_{\alpha} (1+\epsilon) \frac{4}{b} r^{2b} \log(\tfrac{1}{r}) & & \hspace{-2.2cm} \text{if } m_{1}\geq 1 \mbox{ and } 2b= \frac{1}{\alpha}, 	\\
c r^{\frac{1}{\alpha}} & \leq
 \sum_{j=1}^m \nu_j(\partial U_j \cap B_r(0))
\leq C r^{\frac{1}{\alpha}} & & \hspace{-2.2cm} \text{if } m_{1}\geq 1 \mbox{ and } 2b > \frac{1}{\alpha}.
\end{align*}
Combining these inequalities with Lemma \ref{decouplinglemma}, the asymptotic estimate \eqref{numultiplebounds} follows. The estimate \eqref{nuestimate2} is a direct consequence of \eqref{numultiplebounds} (see also \cite[end of Section 5]{CL Corner}). This completes the proof of Theorem \ref{mainth2}.

\section{Balayage of the uniform measure: proof of Theorem \ref{thm:EG tacnode nu}}\label{uniformsec}

Let $c>a>0$ and let $\Omega_{a,c}$ be the tacnocal region defined in (\ref{Omegaacdef}). Let $\mu_{b}$ be the measure defined by $d\mu_{b}(z) = \frac{b^{2}}{\pi}|z|^{2b-2}d^{2}z$. Even though our final result (Theorem \ref{thm:EG tacnode nu}) only applies to the uniform measure (corresponding to $b=1$), we first consider the more general case $b>0$.

The balayage measure $\nu := \mathrm{Bal}(\mu_{b},\partial \Omega_{a,c})$ can be expressed in terms of the harmonic measure as in \eqref{nuharmonicmeasure}. Equivalently, using e.g. \cite[Theorem 4.6]{C2023} (see also \cite[Corollary II.2.6]{GM2005} and \cite[Theorem II.4.11]{SaTo}), $\nu$ can also be expressed in terms of the Green function $g_{\Omega_{a,c}}$ of $\Omega_{a,c}$ as follows:
\begin{align}\label{general formula for dnu in terms of green function}
d\nu(z) =  \bigg( \int_{\Omega_{a,c}} \frac{\partial g_{\Omega_{a,c}}(w,z)}{\partial \mathbf{n}_{z}}d\mu_{b}(w) \bigg)|dz| \qquad \mbox{for all } z\in (\partial \Omega_{a,c})\setminus \{0\},
\end{align}
where $\mathbf{n}_{z}$ is the normal at $z$ pointing into $\Omega_{a,c}$. Recall that the Green function $g_{\mathbb{H}}$ of the upper half-plane is given by
\begin{align}\label{gCplus}
g_{\mathbb{H}}(z,w) = \frac{1}{2\pi} \log \bigg| \frac{z-\overline{w}}{z-w} \bigg|.
\end{align}
The function $\varphi_{1}(z) = i \frac{c}{c-a}\frac{z-2ai}{z}$ maps $\Omega_{a,c}$ to the strip $\{z:\im z \in (0,1)\}$. Hence 
\begin{align}\label{mapping general d to half plane}
\varphi(z) := e^{\pi \varphi_{1}(z)} = \exp \bigg( \pi i \frac{c}{c-a}\frac{z-2ai}{z} \bigg),
\end{align}
is a conformal map from $\Omega_{a,c}$ onto the upper half-plane $\mathbb{H} =\{z:\im z >0\}$, see Figure \ref{fig:conformal map from curvi triangle to upper half plane}. 
\begin{figure}
\begin{center}
\begin{tikzpicture}[master]
\node at (0,0) {\includegraphics[height=4.5cm]{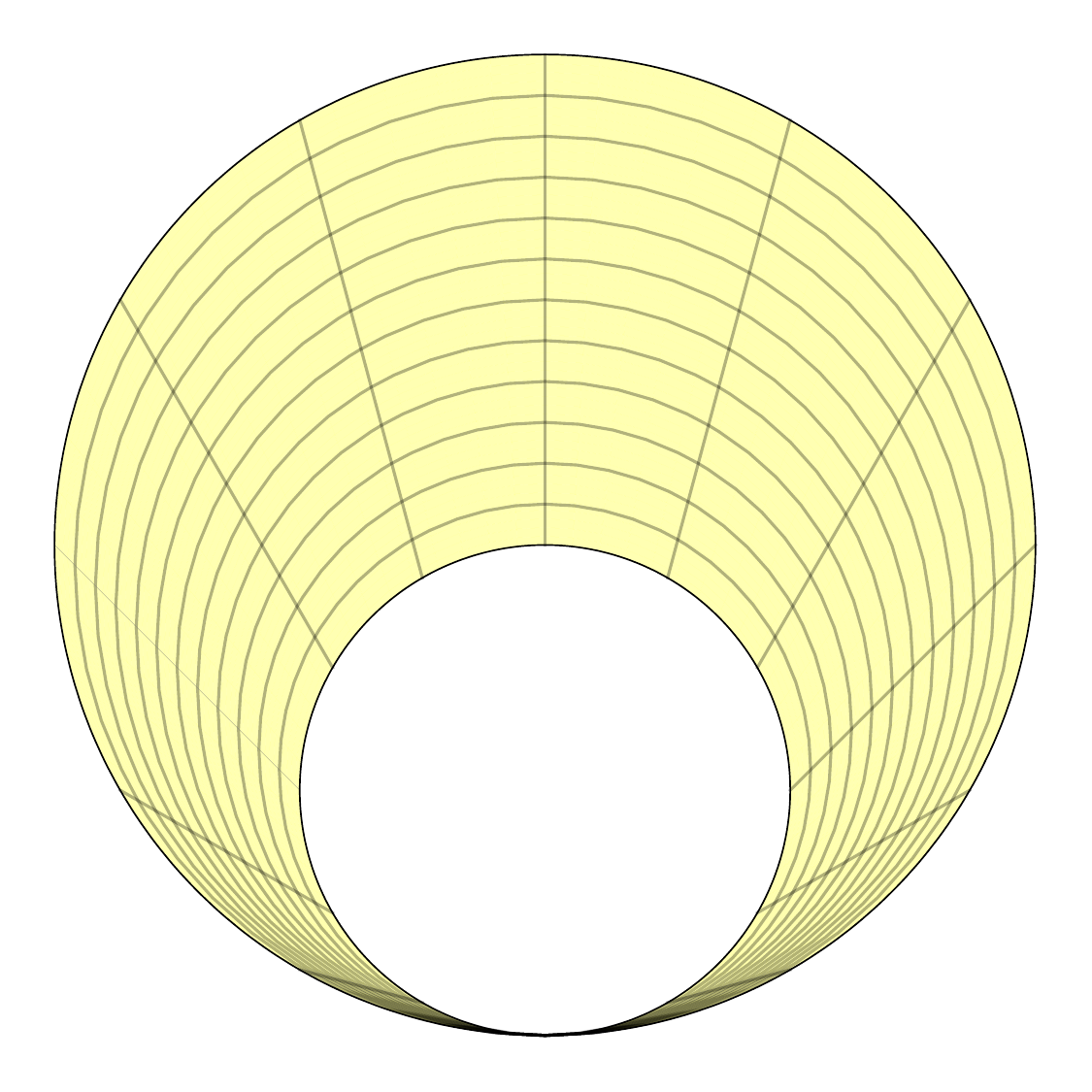}};
\draw[fill] (0,-2.02) circle (0.04);
\node at (0,-2.22) {\footnotesize $0$};
\draw[fill] (0,-1.27) circle (0.04);
\node at (0,-1.47) {\footnotesize $ia$};
\draw[fill] (0,0) circle (0.04);
\node at (0,-0.2) {\footnotesize $ic$};
\end{tikzpicture} \hspace{2.5cm}
\begin{tikzpicture}[slave]
\node at (0,0) {\includegraphics[height=3cm]{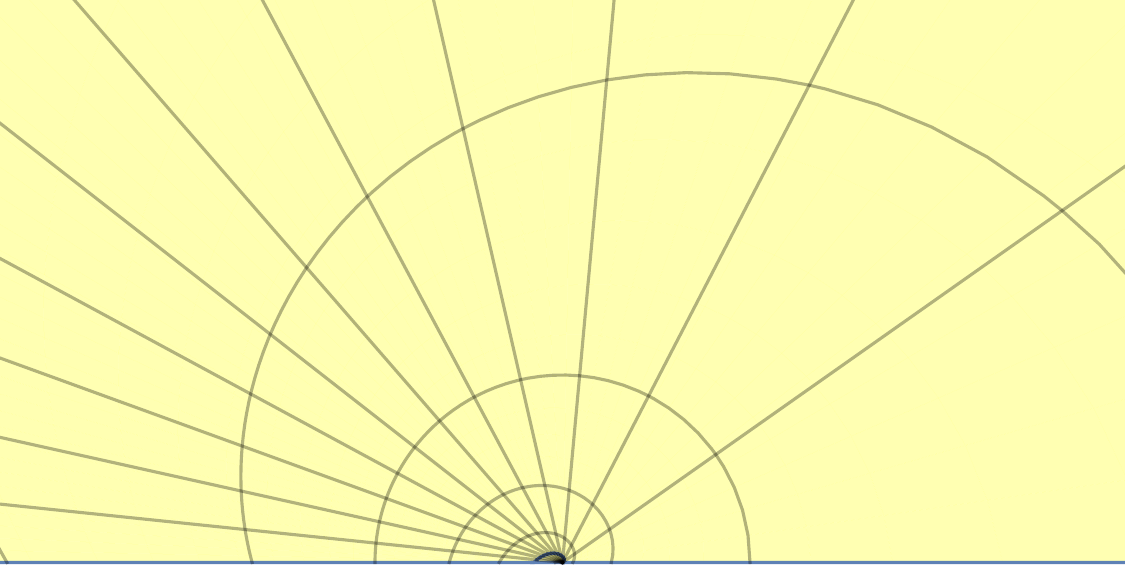}};
\node at (-4.2,1.4) {$\varphi$};
\coordinate (P) at ($(-4.2, 0.5) + (30:1cm and 0.6cm)$);
\draw[thick, black, -<-=0.45] ($(-4.2, 0.5) + (30:1cm and 0.6cm)$(P) arc (30:150:1cm and 0.6cm);
\end{tikzpicture}
\end{center}
\caption{\label{fig:conformal map from curvi triangle to upper half plane} The conformal map $\varphi$ from $\Omega_{a,c}$ onto the upper half-plane, with $a=0.45$ and $c=1.2$.}
\end{figure}
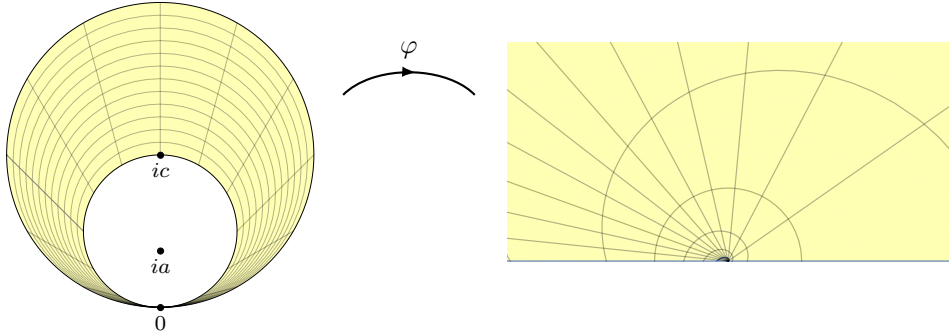
Using (\ref{gCplus}) and \eqref{mapping general d to half plane}, we obtain
\begin{align*}
g_{\Omega_{a,c}}(z,w) = \frac{1}{2\pi} \log \bigg| \frac{\varphi(z)-\overline{\varphi(w)}}{\varphi(z)-\varphi(w)} \bigg|.
\end{align*}
Note that $\Omega_{a,c}$ can be parametrized as
\begin{align}\label{param of U}
\Omega_{a,c}=\{ai+re^{i\theta}: \theta \in (-\tfrac{\pi}{2},\tfrac{3\pi}{2}), r\in (a,R_{e}(\theta))\}
\end{align}
where 
\begin{align}\label{def of Re}
R_{e}(\theta) = (c-a)\sin \theta + \sqrt{a(2c-a)(\cos \theta)^{2} + c^{2} (\sin \theta)^{2}}.
\end{align}
For $w=ai+xe^{i\alpha}$, $\alpha \in (-\tfrac{\pi}{2},\tfrac{3\pi}{2})$, $x\in (a,R_{e}(\theta))$, we have $d^{2}w = xdxd\alpha$, and thus
\begin{align*}
d\mu_{b}(ai+xe^{i\alpha}) = \frac{b^{2}}{\pi}|ai+xe^{i\alpha}|^{2b-2} xdxd\alpha.
\end{align*}
Using \eqref{general formula for dnu in terms of green function}, we obtain, for $z = ai+ae^{i\theta}$ and $\theta \in (-\frac{\pi}{2},\frac{3\pi}{2})$,
\begin{align*}
 \frac{d\nu(z)}{|dz|} & = \int_{-\frac{\pi}{2}}^{\frac{3\pi}{2}}  \int_{a}^{R_{e}(\alpha)}  \bigg[\frac{d}{dr} g_{\Omega_{a,c}}(ai+x  e^{i\alpha},ai+re^{i\theta}) \bigg]_{r=a} d\mu_{b}(ai+xe^{i\alpha}) \nonumber 
 	\\
& = \re \frac{2ce^{i\theta}}{a(c-a)(i+e^{i\theta})^{2}} \int_{-\frac{\pi}{2}}^{\frac{3\pi}{2}}  \int_{a}^{R_{e}(\alpha)}  \frac{1}{1-e^{\frac{2ic\pi}{c-a}\frac{e^{i\alpha}x-ae^{i\theta}}{(i+e^{i\theta})(a-ie^{i\alpha}x)}}} \frac{b^{2}}{\pi}|ai+xe^{i\alpha}|^{2b-2}xdxd\alpha. 
\end{align*}
From \eqref{param of U}, we conclude that $w\in \Omega_{a,c}$ whenever $w=ai+xe^{i\alpha}$ for some $\alpha \in (-\frac{\pi}{2},\frac{3\pi}{2})$ and $x\in (a,R_{e}(\alpha))$. Hence, the above can be rewritten as
\begin{align}\label{lol1}
& \frac{d\nu(z)}{|dz|} = \re \frac{2ce^{i\theta}}{a(c-a)(i+e^{i\theta})^{2}} \int_{\Omega_{a,c}}  \frac{1}{1-e^{\frac{2c\pi}{c-a}\frac{ai+ae^{i\theta}-w}{(i+e^{i\theta})w}}} \frac{b^{2}}{\pi} w^{b-1}\overline{w}^{b-1}d^{2}w, 
\end{align}
where the principal branch of the root is used if $b\notin \N_{>0}$.

The denominator of the integrand vanishes if and only if $w=p_{k}$ for some $k\in \Z$, where $p_{k}=p_{k}(\theta)$ is given by \eqref{def of pk and qk}. A direct analysis shows that either $|p_{k}-ai|\leq a$ or $|p_{k}-ci|\geq c$ holds, so that $p_{k}\notin \Omega_{a,c}$, see also Figure \ref{fig:poles} (left). Note also that $|p_{0}-ai|= a$.

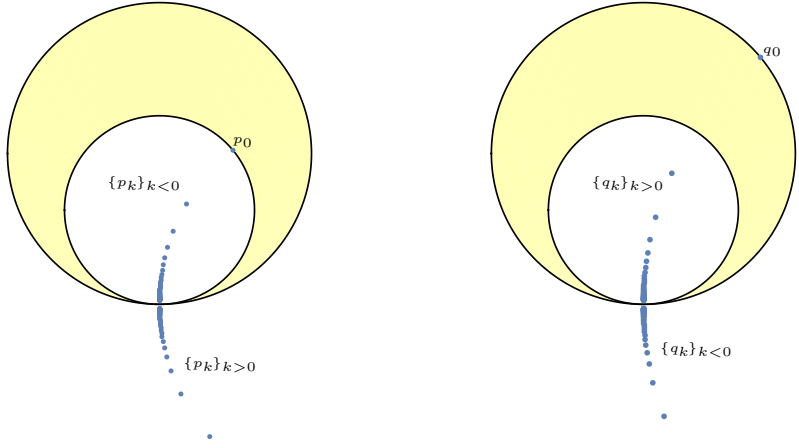
\begin{figure}[h]
\begin{center}
\begin{tikzpicture}[master]
\node at (0,0) {\includegraphics[height=6cm]{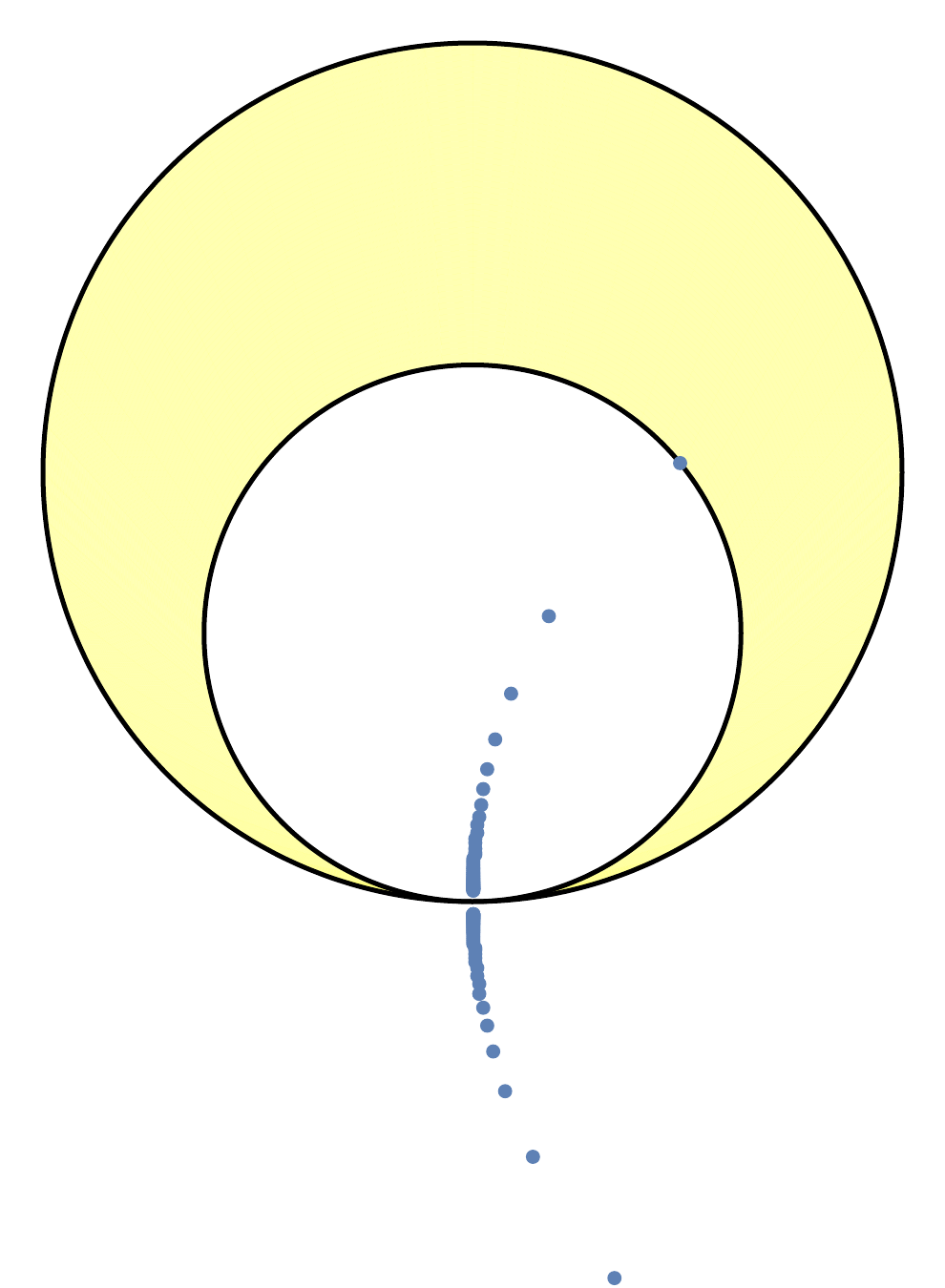}};
\node at (-0.2,0.4) {\tiny $\{p_{k} \hspace{-0.03cm}\}_{k<0}$};
\node at (0.8,-2) {\tiny $\{p_{k} \hspace{-0.03cm}\}_{k>0}$};
\node at (1.1,0.95) {\tiny $p_{0}$};
\end{tikzpicture} \hspace{1.5cm}
\begin{tikzpicture}[slave]
\node at (0,0) {\includegraphics[height=6cm]{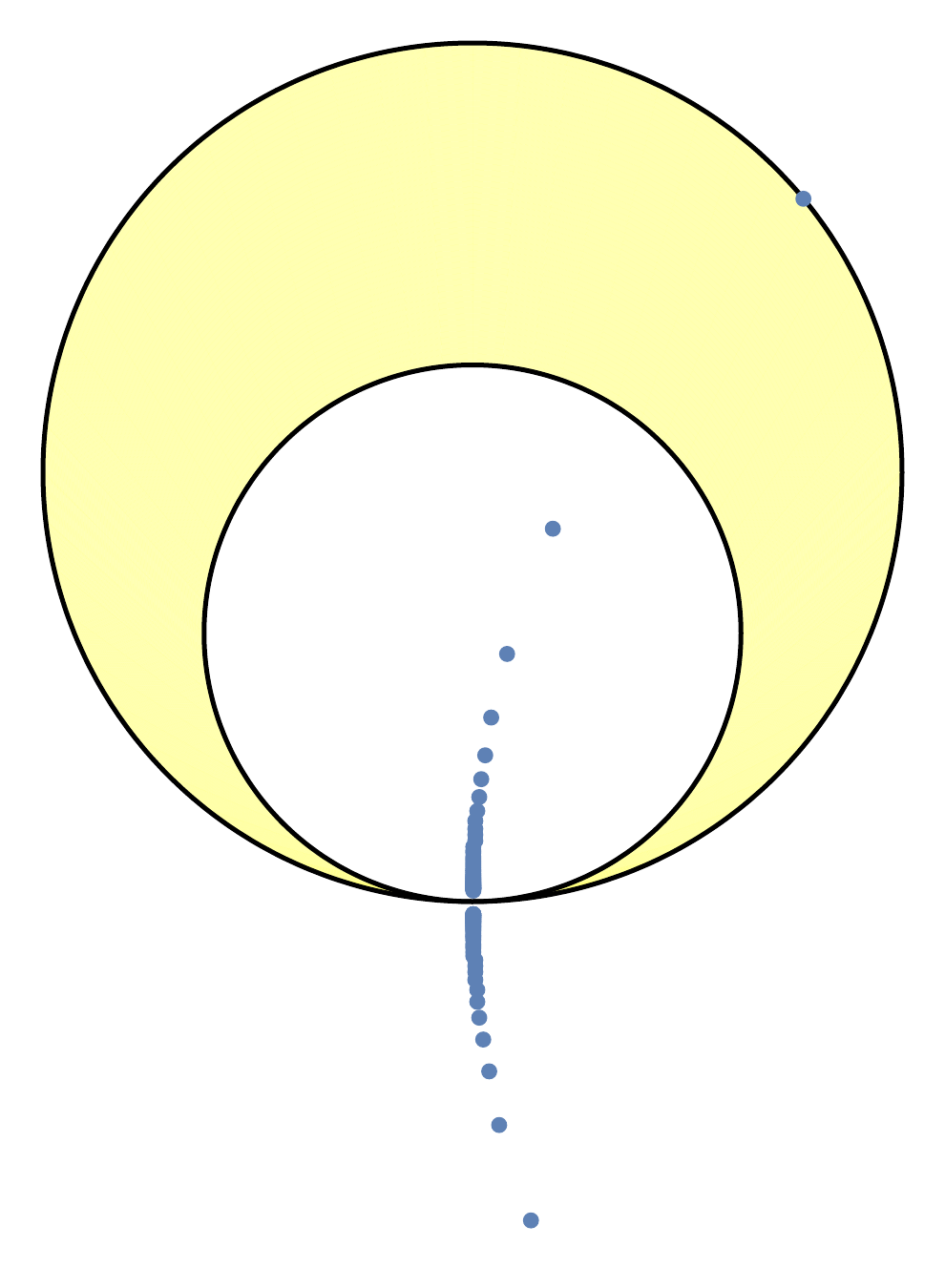}};
\node at (-0.2,0.4) {\tiny $\{q_{k} \hspace{-0.03cm}\}_{k>0}$};
\node at (1.7,2.15) {\tiny $q_{0}$};
\node at (0.7,-1.8) {\tiny $\{q_{k} \hspace{-0.03cm}\}_{k<0}$};
\end{tikzpicture}
\end{center}
\caption{\label{fig:poles} The points $\{p_{k}\}_{k\in \Z}$ (left) and $\{q_{k}\}_{k\in \Z}$ (right) lie outside $\Omega_{a,c}$ (yellow). (In the figure, $\theta = 0.22\pi$, $a=0.75$ and $c=1.2$.)}
\end{figure}

Since the integrand in \eqref{lol1} has no singularity in $\Omega_{a,c}$, we can apply Green's theorem. However, since the integrand has an essential singularity at $0\in \partial \Omega_{a,c}$ and a simple pole at $p_{0}\in \partial \Omega_{a,c}$, extra care must be taken. Using Green's theorem, we obtain
\begin{align}
& \int_{\Omega_{a,c}}  \frac{1}{1-e^{\frac{2c\pi}{c-a}\frac{ai+ae^{i\theta}-w}{(i+e^{i\theta})w}}} \frac{b^{2}}{\pi} w^{b-1}\overline{w}^{b-1}d^{2}w 
= \lim_{\epsilon\to 0_{+}} \int_{\Omega_{a+\epsilon,c-\epsilon}\setminus B_{\epsilon}(0)}  \frac{1}{1-e^{\frac{2c\pi}{c-a}\frac{ai+ae^{i\theta}-w}{(i+e^{i\theta})w}}} \frac{b^{2}}{\pi} w^{b-1}\overline{w}^{b-1}d^{2}w 
	\nonumber \\
& = \frac{1}{2i} \lim_{\epsilon\to 0_{+}} \int_{\partial \big(\Omega_{a+\epsilon,c-\epsilon}\setminus B_{\epsilon}(0)\big)} \frac{1}{1-e^{\frac{2c\pi}{c-a}\frac{ai+ae^{i\theta}-w}{(i+e^{i\theta})w}}} \frac{b^{2}}{\pi} w^{b-1}\frac{\overline{w}^{b}}{b}dw 
	\nonumber \\
& = \frac{b}{2\pi i} \bigg( \int_{C_{c-0_{+}}(ci)} - \int_{C_{a+0_{+}}(ai)} \bigg) \frac{1}{1-e^{\frac{2c\pi}{c-a}\frac{ai+ae^{i\theta}-w}{(i+e^{i\theta})w}}} w^{b-1} \overline{w}^{b} dw \label{lol2}
\end{align}
where $C_{r}(z_{0})$ denotes the circle of radius $r$ centered at $z_{0}$ and oriented in the counterclockwise direction. In the last equality above, we have used that 
\begin{align*}
\frac{1}{2i} \lim_{\epsilon\to 0} \int_{\partial \big(\Omega_{a+\epsilon,c-\epsilon}\setminus B_{\epsilon}(0)\big) \cap \overline{B_{\epsilon}(0)}} \frac{1}{1-e^{\frac{2c\pi}{c-a}\frac{ai+ae^{i\theta}-w}{(i+e^{i\theta})w}}} \frac{b^{2}}{\pi} w^{b-1}\frac{\overline{w}^{b}}{b}dw = \lim_{\epsilon\to 0} \bigO(\epsilon^{2b}) = 0,
\end{align*}
which follows from the fact that 
\begin{align*}
\frac{2c\pi}{c-a}\frac{ai+ae^{i\theta}-w}{(i+e^{i\theta})w} = \frac{2ac\pi}{(c-a)w} + \bigO(1), \qquad \mbox{as } w\to 0.
\end{align*}

We assume from now on that $b=1$ (it is in principle possible to obtain explicit results for the more general case $b\in \N_{>0}$, but it becomes technically more involved).
For $w\in C_{r}(ri)$, we have $\overline{w}= \frac{-irw}{w-ri}$, so the right-hand side of \eqref{lol2} can be rewritten as
\begin{multline}\label{lol3}
\frac{1}{2\pi i} \int_{C_{c-0_{+}}(ci)}\frac{1}{1-e^{\frac{2c\pi}{c-a}\frac{ai+ae^{i\theta}-w}{(i+e^{i\theta})w}}} \frac{-ciw}{w-ci} dw 
- \frac{1}{2\pi i} \int_{C_{a+0_{+}}(ai)}\frac{1}{1-e^{\frac{2c\pi}{c-a}\frac{ai+ae^{i\theta}-w}{(i+e^{i\theta})w}}} \frac{-aiw}{w-ai} dw.
\end{multline}

As $w\to p_{k}$, $k\in \Z$,
\begin{align*}
\frac{2c\pi}{c-a}\frac{ai+ae^{i\theta}-w}{(i+e^{i\theta})w} = 2\pi i k + \frac{1}{\alpha_{k}} (w-p_{k}) + \bigO((w-p_{k})^{2}),
\end{align*}
where 
\begin{align}\label{alphakdef}
\alpha_{k}=\alpha_{k}(\theta) := \frac{ac(c-a)(i+e^{i\theta})^{2}}{2\pi (ic-k(c-a)(i+e^{i\theta}))^{2}}.
\end{align}
Hence, for any $k \in \Z$,
\begin{align*}
\mbox{Res}\bigg( \frac{1}{1-e^{\frac{2c\pi}{c-a}\frac{ai+ae^{i\theta}-w}{(i+e^{i\theta})w}}} \frac{\frac{-i r w}{w-r i}}{\pi}  ; w=p_{k} \bigg) = \frac{\alpha_{k}}{\pi} \frac{ri p_{k}}{p_{k}-ri}, \qquad r\in \{a,c\}.
\end{align*}
Also, a computation yields
\begin{align*}
& -\mbox{Res}\bigg( \frac{1}{1-e^{\frac{2c\pi}{c-a}\frac{ai+ae^{i\theta}-w}{(i+e^{i\theta})w}}}  \frac{-aiw}{w-ai}  ; w=\infty \bigg) = \frac{a^{2}}{1-e^{-cG(\theta)}} - \frac{2 \pi ia^{2}ce^{cG(\theta)}}{(c-a)(1-e^{cG(\theta)})^{2}}, \\
& \mbox{Res}\bigg( \frac{1}{1-e^{\frac{2c\pi}{c-a}\frac{ai+ae^{i\theta}-w}{(i+e^{i\theta})w}}} \frac{-ciw}{w-ci}  ; w=ci \bigg) =  \frac{c^{2}}{1-e^{\frac{2\pi i}{c-a}\frac{ci-ai-ae^{i\theta}}{i+e^{i\theta}}}},
\end{align*}
where $G(\theta):=\frac{2\pi}{(c-a)(i+e^{i\theta})}$. Deforming $C_{a+0_{+}}(ai)$ toward $\infty$ and $C_{c-0_{+}}(ci)$ toward $ci$ in \eqref{lol3}, and using \eqref{lol1} and $\frac{e^{i\theta}}{(i+e^{i\theta})^{2}} = \frac{-i}{2(1+\sin \theta)}$, we obtain, for $z = ai+ae^{i\theta}$ and $\theta \in (-\frac{\pi}{2},\frac{3\pi}{2})$, 
\begin{align}
& \frac{d\nu(z)}{d\theta} = \frac{c}{(c-a)(1+\sin \theta)} \im \bigg\{ \frac{c^{2}}{1-e^{\frac{2\pi i}{c-a}\frac{ci-ai-ae^{i\theta}}{i+e^{i\theta}}}} - \frac{a^{2}}{1-e^{-cG(\theta)}}
+\frac{2 \pi ia^{2}ce^{cG(\theta)}}{(c-a)(1-e^{cG(\theta)})^{2}} \nonumber \\
&\hspace{4.7cm}  + \sum_{k=-\infty}^{0} \alpha_{k} \frac{cip_{k}}{p_{k}-ci} + \sum_{k=1}^{+\infty} \alpha_{k} \frac{aip_{k}}{p_{k}-ai} \bigg\}. \label{nu on Ca in proof}
\end{align}
After further simplifications, we get \eqref{nu on Ca}. 

For the computation of $d\nu(z)$ with $z = ci+ce^{i\theta}$, $\theta \in (-\frac{\pi}{2},\frac{3\pi}{2})$, instead of using the conformal map $\varphi$ defined in \eqref{mapping general d to half plane}, it is more natural to use
\begin{align*}
\phi(z) := \exp \bigg( \pi i \frac{-a}{c-a}\frac{z-2ci}{z} \bigg).
\end{align*}
This function is also a conformal map from $\Omega_{a,c}$ onto $\mathbb{H}$. The formula \eqref{general formula for dnu in terms of green function} yields
\begin{align*}
\frac{d\nu(z)}{|dz|} = \re \frac{2ae^{i\theta}}{c(c-a)(i+e^{i\theta})^{2}} \iint_{\substack{x>0,\alpha\in (-\frac{\pi}{2},\frac{3\pi}{2}): \\ ci+xe^{i\alpha}\in \Omega_{a,c}}}  \frac{1}{1-e^{-\frac{2ia\pi}{c-a}\frac{e^{i\alpha}x-ce^{i\theta}}{(i+e^{i\theta})(c-ie^{i\alpha}x)}}} \frac{b^{2}}{\pi}|ci+xe^{i\alpha}|^{2b-2}xdxd\alpha.
\end{align*}
For $b=1$, a similar analysis as above yields 
\begin{align*}
& \frac{d\nu(z)}{d\theta} = \frac{a}{(c-a)(1+\sin \theta)} 
\im \bigg\{ \frac{c^{2}}{1-e^{\frac{2\pi i}{c-a}\frac{a e^{i\theta}}{i+e^{i\theta}}}} - \frac{a^{2}}{1-e^{aG(\theta)}} 
-\frac{2 \pi ia^{2}ce^{aG(\theta)}}{(c-a)(1-e^{aG(\theta)})^{2}}
	\\
&\hspace{4.7cm}  +  \sum_{k=-\infty}^{0} \beta_{k} \frac{aiq_{k}}{q_{k}-ai} +  \sum_{k=1}^{+\infty} \beta_{k} \frac{ciq_{k}}{q_{k}-ci} \bigg\},
\end{align*}
where $\beta_{k}=\beta_{k}(\theta) := \frac{-ac(c-a)(i+e^{i\theta})^{2}}{2\pi (ia+k(c-a)(i+e^{i\theta}))^{2}}$ and $q_{k}$ is given by \eqref{def of pk and qk}, and after further simplifications we obtain \eqref{nu on Cc}.
This completes the proof of the first part of Theorem \ref{thm:EG tacnode nu}.

It remains to prove \eqref{asymp of dnu near 0 uniform 1}, \eqref{asymp of dnu near 0 uniform 2}, and \eqref{asymp of nu near 0 uniform}. From \eqref{nu on Ca}, we infer that
\begin{align}
& \frac{d\nu(ai+ae^{i\theta})}{d\theta} = \re \bigg[ \frac{ac^{3}}{\pi} \sum_{k=-\infty}^{0} f_{k}(\theta) \frac{ip_{k}(\theta)}{p_{k}(\theta)-ci} + \frac{a^{2}c^{2}}{\pi} \sum_{k=1}^{+\infty} f_{k}(\theta) \frac{ip_{k}(\theta)}{p_{k}(\theta)-ai} \bigg] + \bigO(e^{-\frac{c\pi}{c-a}\frac{1}{\theta+\frac{\pi}{2}}}) \nonumber \\
& = \frac{a^{2}(3c-2a)}{2c\pi}(\theta + \tfrac{\pi}{2})^{2} +   \sum_{k=1}^{+\infty} \re \bigg[ \frac{ac^{3}}{\pi} f_{-k}(\theta) \frac{ip_{-k}(\theta)}{p_{-k}(\theta)-ci} + \frac{a^{2}c^{2}}{\pi} f_{k}(\theta) \frac{ip_{k}(\theta)}{p_{k}(\theta)-ai} \bigg] + \bigO\big((\theta + \tfrac{\pi}{2})^{4}\big) \label{lol18}
\end{align}
as $\theta \searrow -\frac{\pi}{2}$. Let $\lambda := k(\theta + \frac{\pi}{2})$. 
We observe from (\ref{def of pk and qk}) that, for $j = 0,1, 2, 3, 4$,
\begin{align}\label{fkpkbounds}
\bigg|\frac{\partial^j}{\partial\theta^j}\Big|_{\lambda} f_k(\theta)\bigg| \leq \frac{C}{1 + \lambda^2}, \qquad \bigg|\frac{\partial^j}{\partial\theta^j}\Big|_{\lambda} \frac{p_k(\theta)}{p_{k}(\theta)-a i}\bigg| \leq \frac{C}{1 + |\lambda|}, \qquad \bigg|\frac{\partial^j}{\partial\theta^j}\Big|_{\lambda} \frac{p_k(\theta)}{p_{k}(\theta)-c i}\bigg| \leq \frac{C}{1 + |\lambda|}
\end{align}
for all $\theta > -\frac{\pi}{2}$ sufficiently close to $-\frac{\pi}{2}$ and all $k \in \R$, where $\frac{\partial}{\partial\theta}|_{\lambda}$ denotes the partial derivative with respect to $\theta$ with $\lambda$ held fixed. Taylor expanding around $\theta = -\frac{\pi}{2}$ with $\lambda \geq 0$ held fixed, and then using (\ref{fkpkbounds}) to bound the error term uniformly in $\lambda$, we obtain, as $\theta \searrow -\frac{\pi}{2}$,
\begin{multline*}
\re \bigg[ \frac{ac^{3}}{\pi} f_{-k}(\theta) \frac{ip_{-k}(\theta)}{p_{-k}(\theta)-ci} + \frac{a^{2}c^{2}}{\pi} f_{k}(\theta) \frac{ip_{k}(\theta)}{p_{k}(\theta)-ai} \bigg] \\
= h_{2}(\lambda) (\theta + \tfrac{\pi}{2})^{2} + h_{3}(\lambda) (\theta + \tfrac{\pi}{2})^{3} + \bigO \bigg( \frac{(\theta + \tfrac{\pi}{2})^{4}}{1+\lambda^{3}} \bigg)
\end{multline*}
uniformly for $\lambda \geq 0$, where
\begin{align*}
& h_{2}(\lambda) = \frac{a^{2}(2c-a)c^{3}\big( c^{4}-6(c-a)^{2}c^{2}\lambda^{2}+(c-a)^{4}\lambda^{4} \big)}{\pi(c^{2}+(c-a)^{2}\lambda^{2})^{4}}, \\
& h_{3}(\lambda) = -\frac{a^{2}(c-a)^{3}c^{3}\lambda\big( 5c^{4}-10(c-a)^{2}c^{2}\lambda^{2}+(c-a)^{4}\lambda^{4} \big)}{\pi(c^{2}+(c-a)^{2}\lambda^{2})^{5}}.
\end{align*}
Using the Euler--MacLaurin formula, as $\theta \searrow -\frac{\pi}{2}$ we get
\begin{multline*}
\sum_{k=1}^{+\infty} \re \bigg[ \frac{ac^{3}}{\pi} f_{-k}(\theta) \frac{ip_{-k}(\theta)}{p_{-k}(\theta)-ci} + \frac{a^{2}c^{2}}{\pi} f_{k}(\theta) \frac{ip_{k}(\theta)}{p_{k}(\theta)-ai} \bigg] = (\theta + \tfrac{\pi}{2})^{2} \bigg[ \frac{1}{\theta + \frac{\pi}{2}}\int_{0}^{\infty} h_{2}(\lambda)d\lambda \\  - \frac{h_{2}(0)}{2} + \bigO(\theta + \tfrac{\pi}{2}) \bigg] + (\theta + \tfrac{\pi}{2})^{3} \bigg[ \frac{1}{\theta + \frac{\pi}{2}}\int_{0}^{\infty} h_{3}(\lambda)d\lambda + \bigO(1) \bigg] + \bigO\big( (\theta + \tfrac{\pi}{2})^{3} \big).
\end{multline*}
Since
\begin{align*}
\int_{0}^{\infty} h_{2}(\lambda)d\lambda = 0, \qquad h_{2}(0) = \frac{a^{2}(2c-a)}{c\pi}, \qquad \int_{0}^{\infty} h_{3}(\lambda)d\lambda = -\frac{a^{2}(c-a)}{4c\pi},
\end{align*}
we have
\begin{align*}
\sum_{k=1}^{+\infty} \re \bigg[ \frac{ac^{3}}{\pi} f_{-k}(\theta) \frac{ip_{-k}(\theta)}{p_{-k}(\theta)-ci} + \frac{a^{2}c^{2}}{\pi} f_{k}(\theta) \frac{ip_{k}(\theta)}{p_{k}(\theta)-ai} \bigg] = -\frac{a^{2}(5c-3a)}{4c\pi}(\theta + \tfrac{\pi}{2})^{2} + \bigO\big( (\theta + \tfrac{\pi}{2})^{3} \big),
\end{align*}
and substituting the above into \eqref{lol18} yields \eqref{asymp of dnu near 0 uniform 1}. The proof of \eqref{asymp of dnu near 0 uniform 2} is similar. 

We now turn to the proof of \eqref{asymp of nu near 0 uniform}. The first identity in \eqref{asymp of nu near 0 uniform} directly follows from the symmetry
\begin{align*}
\frac{d\nu(ai+ae^{i(-\frac{\pi}{2}+\theta)})}{d\theta} = \frac{d\nu(ai+ae^{i(\frac{3\pi}{2}-\theta)})}{d\theta}, \qquad \frac{d\nu(ci+ce^{i(-\frac{\pi}{2}+\theta)})}{d\theta} = \frac{d\nu(ci+ce^{i(\frac{3\pi}{2}-\theta)})}{d\theta},
\end{align*}
which can be verified using \eqref{nu on Ca} and \eqref{nu on Cc}. For all $r \in (0, 2a]$, 
\begin{align}\label{lol19}
\nu\big(\partial \Omega \cap B_r(0) \cap \{z:\re z>0\}\big) = \int_{-\frac{\pi}{2}}^{\theta_{a}(r)} d\nu(ai+ae^{i\theta})  + \int_{-\frac{\pi}{2}}^{\theta_{c}(r)} d\nu(ci+ce^{i\theta}),
\end{align}
where
\begin{align*}
\theta_{a}(r) = \arcsin \bigg( \frac{r^2}{2a^2} -1 \bigg), \qquad \theta_{c}(r) = \arcsin \bigg( \frac{r^2}{2c^2} -1 \bigg).
\end{align*}
Taking $r\to 0$ in \eqref{lol19}, and using \eqref{asymp of dnu near 0 uniform 1}--\eqref{asymp of dnu near 0 uniform 2}, we find \eqref{asymp of nu near 0 uniform}. This finishes the proof of Theorem \ref{thm:EG tacnode nu}.

\section{Hole probabilities: proof of Theorem \ref{thm:EG tacnode C}}\label{holesec}

Let $\beta>0$ and $\tau\in [0,1)$. Let $S$ and $\mu$ be as in (\ref{Sdef}) and \eqref{mu S EG beginning of intro}. Let $z_{0}\in \C$, $\theta_{0}\in (-\pi,\pi]$, and $0<a<c$ be such that $z_0 + e^{i\theta_0}\Omega_{a,c} \subset S$. By \cite[Theorem 2.5 (i)]{C2023}, as $n \to +\infty$ we have $\mathbb{P}(\# \{z_{j}\in z_0 + e^{i\theta_0}\Omega_{a,c}\} = 0) = \exp \big( -Cn^{2}+o(n^{2}) \big)$, where 
\begin{align}\label{C in computation}
C = \frac{\beta}{4}\bigg(  \int_{\partial \Omega_{a,c}}Q(z)d\nu(z) - \frac{1}{1-\tau^{2}} \int_{\Omega_{a,c}}Q(z)\frac{d^{2}z}{\pi} \bigg),
\end{align}
$Q(z) = \frac{1}{1-\tau^{2}}\big( |z|^{2}-\tau \, \re z^{2} \big)$, and $\nu = \mathrm{Bal}(\mu|_{\Omega_{a,c}},\partial \Omega_{a,c})$. 

By \cite[Theorem 2.24]{C2023}, $C=C(\tau,z_{0},\theta_{0})$ satisfies
\begin{align}\label{Cscaling}
C(\tau,z_{0},\theta_{0})=\frac{1}{(1-\tau^{2})^{2}}C(0,0,0).
\end{align}
Hence in what follows we assume that $z_{0}=0$, $\theta_{0}=0$, and $\tau=0$.

\subsection{Evaluation of $\int_{\Omega_{a,c}}Q(z)\frac{d^{2}z}{\pi}$}
Let us compute the second integral in (\ref{C in computation}). Using the parametrization \eqref{param of U} of $\Omega_{a,c}$, we get
\begin{align}
\int_{\Omega_{a,c}}Q(z)\frac{d^{2}z}{\pi} & = \int_{-\frac{\pi}{2}}^{\frac{3\pi}{2}}\int_{a}^{R_{e}(\theta)} \big( a^{2}+r^{2}+2ar\sin \theta \big) \frac{rdrd\theta}{\pi} \nonumber \\
& = \frac{a^{2}}{2\pi} \int_{-\pi}^{\pi}R_{e}(\theta)^{2}d\theta + \frac{1}{4\pi} \int_{-\pi}^{\pi} R_{e}(\theta)^{4}d\theta + \frac{2a}{3\pi}\int_{-\pi}^{\pi}R_{e}(\theta)^{3}\sin \theta \; d\theta - \frac{3a^{4}}{2}. \label{lol8}
\end{align}

\begin{lemma}
The following identities hold:
\begin{align}
& \int_{-\pi}^{\pi}R_{e}(\theta)^{2}d\theta = 2\pi c^{2}, \label{lol4} \\
& \int_{-\pi}^{\pi} R_{e}(\theta)^{4}d\theta = 2\pi c^{2}(2a^{2}-4ac+3c^{2}), \label{lol5} \\
& \int_{-\pi}^{\pi}R_{e}(\theta)^{3}\sin \theta \; d\theta = 3\pi (c-a)c^{2}. \label{lol6}
\end{align}
\end{lemma}
\begin{proof}
To compute these integrals, we define 
\begin{align*}
\hat{R}_{e}(z) := (c-a) \frac{z-z^{-1}}{2i} + \frac{\sqrt{a^{2}+c^{2}-2ac}}{2} \frac{\sqrt{z_{2}-z}\sqrt{z-z_{1}} \sqrt{z+z_{1}} \sqrt{z+z_{2}} }{z},
\end{align*}
where the principal branch is used for the roots, and
\begin{align*}
z_{2}:= \frac{c+\sqrt{a(2c-a)}}{c-a}, \qquad z_{1} := \frac{c-\sqrt{a(2c-a)}}{c-a}.
\end{align*}
Note that $z_{2}>1$ and $z_{1}\in (0,1)$. So $\hat{R}_{e}$ is analytic in $\C\setminus \big( (-\infty,-z_{2}]\cup [-z_{1},z_{1}] \cup [z_{2},+\infty)\big)$, and has simple poles at $0$ and $\infty$ (on both sides of the cuts). It is easy to check that $\hat{R}_{e}(e^{i\theta}) = R_{e}(\theta)$ for all $\theta \in \R$ (recall that $R_{e}(\theta)$ is defined in \eqref{def of Re}).

For \eqref{lol4}, we apply the change of variables $z=e^{i\theta}$,
\begin{align*}
\int_{-\pi}^{\pi}R_{e}(\theta)^{2}d\theta = \oint_{S_{1}} \hat{R}_{e}(z)^{2} \frac{dz}{iz} = I_{1}+I_{2},
\end{align*}
where $S_{1}$ is the unit circle oriented positively and
\begin{align*}
& I_{1} = \oint_{S_{1}} \bigg\{ \bigg( (c-a) \frac{z-z^{-1}}{2i} \bigg)^{2} + \frac{a^{2}+c^{2}-2ac}{4} \frac{(z_{2}-z)(z-z_{1})(z+z_{1})(z+z_{2})}{z^{2}} \bigg\}  \frac{dz}{iz}, \\
& I_{2} = (c-a)\frac{\sqrt{a^{2}+c^{2}-2ac}}{2} \oint_{S_{1}} (z^{-3}-z^{-1})\sqrt{z_{2}-z}\sqrt{z-z_{1}} \sqrt{z+z_{1}} \sqrt{z+z_{2}} dz.
\end{align*}
The integrand of $I_{1}$ is analytic in $\C\setminus \{0\}$, and can therefore be computed by a residue computation. This gives $I_{1} = 2\pi c^{2}$. For $I_{2}$, by deforming the contour towards $\infty$, we get
\begin{multline*}
\oint_{S_{1}} z^{-3}\sqrt{z_{2}-z}\sqrt{z-z_{1}} \sqrt{z+z_{1}} \sqrt{z+z_{2}} dz = \lim_{M\to +\infty} \bigg[ \oint_{C_{M}(0)} \sqrt{z_{2}-z}\sqrt{z-z_{1}} \sqrt{z+z_{1}} \sqrt{z+z_{2}} \frac{dz}{z^{3}} \\   -2i \bigg(\int_{-M}^{-z_{2}} + \int_{z_{2}}^{M} \bigg) x^{-3}\sqrt{x^{2}-z_{2}^{2}}\sqrt{x^{2}-z_{1}^{2}}dx \bigg] = 0,
\end{multline*}
since the first integral $\oint_{C_{M}(0)} = \oint_{C_{M}(0)\cap \mathbb{H}}+\oint_{C_{M}(0) \setminus \mathbb{H}}$ tends to $\pi-\pi=0$ as $M\to + \infty$, and since the integrand of the second integral is odd. Similarly, by deforming the contour towards $0$, we get
\begin{align}\label{lol7}
\oint_{S_{1}} z^{-1}\sqrt{z_{2}-z}\sqrt{z-z_{1}} \sqrt{z+z_{1}} \sqrt{z+z_{2}} dz = -2i\dashint_{-z_{1}}^{z_{1}}\sqrt{z_{2}^{2}-x^{2}}\sqrt{z_{1}^{2}-x^{2}} \frac{dx}{x} = 0,
\end{align}
where $\dashint$ stands for ``principal value integral". For the first equality in \eqref{lol7}, we have used that
\begin{align*}
\underset{z=0_{+}i}{\res}\bigg( \frac{\sqrt{z_{2}-z}}{z}\sqrt{z-z_{1}} \sqrt{z+z_{1}} \sqrt{z+z_{2}} \bigg) = -\underset{ z=0_{-}i}{\res}\bigg( \frac{\sqrt{z_{2}-z}}{z}\sqrt{z-z_{1}} \sqrt{z+z_{1}} \sqrt{z+z_{2}} \bigg),
\end{align*}
and for the second equality in \eqref{lol7} we have used the fact that the integrand is odd. Thus $I_{2}=0$ and \eqref{lol4} follows.

The proofs of \eqref{lol5} and \eqref{lol6} are similar (albeit more involved), so we omit them.
\end{proof}

Substituting \eqref{lol4}, \eqref{lol5}, and \eqref{lol6} into \eqref{lol8}, we conclude that the second integral in (\ref{C in computation}) is given by
\begin{align}\label{lol17}
\int_{\Omega_{a,c}}Q(z)\frac{d^{2}z}{\pi} = \frac{3}{2}(c^{4}-a^{4}).
\end{align}

\subsection{Evaluation of $\int_{\partial \Omega_{a,c}}Q(z)d\nu(z)$}
We compute the first integral in (\ref{C in computation}). 
Since $\partial \Omega_{a,c} = C_{a}(ia)\cup C_{c}(ic)$, we have
\begin{align}\nonumber
\int_{\partial \Omega_{a,c}}Q(z)d\nu(z) & = \int_{-\frac{\pi}{2}}^{\frac{3\pi}{2}} |ai+ae^{i\theta}|^{2}d\nu(ai+ae^{i\theta}) + \int_{-\frac{\pi}{2}}^{\frac{3\pi}{2}} |ci+ce^{i\theta}|^{2}d\nu(ci+ce^{i\theta}) 
	\\ \label{intOmegaacQnu}
& = 2a^{2} \int_{-\frac{\pi}{2}}^{\frac{3\pi}{2}} (1+\sin \theta) d\nu(ai+ae^{i\theta}) + 2c^{2} \int_{-\frac{\pi}{2}}^{\frac{3\pi}{2}} (1+\sin \theta) d\nu(ci+ce^{i\theta}).
\end{align}
Let $\{\alpha_k\}_{k\in\Z}$ be as in (\ref{alphakdef}). Using \eqref{nu on Ca in proof}, we get
\begin{multline}\label{lol13}
\int_{-\frac{\pi}{2}}^{\frac{3\pi}{2}} (1+\sin \theta) d\nu(ai+ae^{i\theta}) = \frac{c}{c-a}\im \int_{-\frac{\pi}{2}}^{\frac{3\pi}{2}} \bigg\{ \frac{c^{2}}{1-e^{\frac{2\pi i}{c-a}\frac{ci-ai-ae^{i\theta}}{i+e^{i\theta}}}} - \frac{a^{2}}{1-e^{\frac{-2\pi c}{(c-a)(i+e^{i\theta})}}}  \\
+ \frac{2\pi i a^{2}c}{(c-a)\big(e^{\frac{-2\pi c}{(c-a)(i+e^{i\theta})}}+e^{\frac{2\pi c}{(c-a)(i+e^{i\theta})}}-2\big)} + c \sum_{k=-\infty}^{0} \alpha_{k} \frac{ip_{k}}{p_{k}-ci} + a \sum_{k=1}^{+\infty} \alpha_{k} \frac{ip_{k}}{p_{k}-ai} \bigg\} d\theta.
\end{multline}

\begin{lemma}\label{lemma:pkc}
Assume that $\frac{c}{c-a}\notin \N:=\{0,1,\ldots\}$.
\begin{itemize}
\item[(i)] For $k\in \big(\lceil \frac{\frac{c}{2}-a}{c-a} \rceil+\N\big) \cap (-\N)$, we have
\begin{align*}
\int_{-\pi}^{\pi} \alpha_{k}(\theta) \frac{ip_{k}(\theta)}{p_{k}(\theta)-ci}d\theta = \frac{-ia^{2}c}{(c+(c-a)|k|)^{2}(|k|+1)} + \frac{ic^{2}}{(a-(c-a)|k|)(1+|k|)}.
\end{align*}
\item[(ii)] For $k\in \big(\lceil \frac{\frac{c}{2}-a}{c-a} \rceil-1-\N\big) \cap (-\N)$, we have
\begin{align*}
\int_{-\pi}^{\pi} \alpha_{k}(\theta) \frac{ip_{k}(\theta)}{p_{k}(\theta)-ci}d\theta = \frac{-ia^{2}c}{(c+(c-a)|k|)^{2}(|k|+1)}.
\end{align*}
\end{itemize}
\end{lemma}
\begin{proof}
Note that $\frac{c}{c-a}\notin \N$ implies both $\frac{a}{c-a}\notin \N$ and $\frac{\frac{c}{2}-a}{c-a}\notin -\N$.

Let $k \in -\N$.
Recalling the definition (\ref{alphakdef}) of $\alpha_{k}(\theta)$ and the expression \eqref{def of pk and qk} for $p_{k}(\theta)$, and performing the change of variables $z=e^{i\theta}$, we get
\begin{align}\label{lol9}
\int_{-\pi}^{\pi} \alpha_{k}(\theta) \frac{ip_{k}(\theta)}{p_{k}(\theta)-ci}d\theta = \oint_{S_{1}} \frac{ac(c-a)(i+z)^{2}}{2\pi(ic-(c-a)k(i+z))^{2}} \frac{i a(i+z)}{a(i+z)-ci(1+ki(i+z)(1-\frac{a}{c}))} \frac{1}{iz} dz.
\end{align}
If $k \in \{-1,-2, \dots\}$, the first fraction has a double pole at
\begin{align}\label{def m alpha k}
m_{\alpha_{k}}:=-i\bigg( 1+\frac{1}{-k}\frac{c}{c-a} \bigg),
\end{align}
and for any $k \in -\N$ the second fraction has a simple pole at
\begin{align}\label{def m pk}
m_{p_{k}} := -i \bigg( 1+\frac{c}{(c-a)(-k)-a} \bigg).
\end{align}
Note that $|m_{\alpha_{k}}|>1$ for all $k\leq -1$, while $|m_{p_{k}}|<1$ if $|k|<\frac{a-\frac{c}{2}}{c-a}$ and $|m_{p_{k}}|>1$ if $|k|>\frac{a-\frac{c}{2}}{c-a}$. Let $h(z)$ be the integrand of the right-hand side of \eqref{lol9}. By deforming the contour towards $0$ in  \eqref{lol9}, and using 
\begin{align*}
\mbox{Res}\big( h(z);z=m_{p_{k}} \big) = \frac{-ic^{2}}{2\pi i (k-1)(a+(c-a)k)}, \quad \mbox{Res}\big( h(z);z=0 \big) = \frac{-ia^{2}c}{2\pi i(c+(c-a)|k|)^{2}(1+|k|)},
\end{align*}
we find the claim. 
\end{proof}

\begin{lemma}\label{lemma:pka}
Assume that $\frac{\frac{c}{2}}{c-a}\notin \N$. 
\begin{itemize}
\item[(i)] For $k\in \{1,2,\ldots,\lceil \frac{\frac{c}{2}}{c-a} \rceil-1\}$, we have
\begin{align*}
\int_{-\pi}^{\pi} \alpha_{k}(\theta) \frac{ip_{k}(\theta)}{p_{k}(\theta)-ai}d\theta = \frac{i a c^2}{(c-a)k^{2}(c+(c-a)k)} - \frac{i ac\big(c^{2}-3(c-a)ck+(c-a)^{2}k^{2}\big)}{(c-a)k^{2}(c-(c-a)k)^{2}}.
\end{align*}
\item[(ii)] For $k\in \{\lceil \frac{\frac{c}{2}}{c-a} \rceil,\lceil \frac{\frac{c}{2}}{c-a} \rceil+1,\ldots\}$, we have
\begin{align*}
\int_{-\pi}^{\pi} \alpha_{k}(\theta) \frac{ip_{k}(\theta)}{p_{k}(\theta)-ai}d\theta = \frac{i a c^2}{(c-a)k^{2}(c+(c-a)k)}.
\end{align*}
\end{itemize}
\end{lemma}
\begin{proof}
The proof is similar to the proof of Lemma \ref{lemma:pkc}.
\end{proof}
\begin{lemma}\label{lemma:exp1}
Assume that $\frac{c}{c-a}\notin \N$. Then
\begin{align}\label{lol10}
\int_{-\pi}^{\pi} \frac{1}{1-e^{\frac{-2\pi c}{(c-a)(i+e^{i\theta})}}} d\theta = \pi + i \pi \cot \Big( \frac{c\pi}{c-a} \Big) + \frac{ci}{c-a} \sum_{k=\lceil \frac{\frac{c}{2}}{c-a} \rceil}^{+\infty} \frac{1}{k}\frac{1}{k-\frac{c}{c-a}}.
\end{align}
\end{lemma}
\begin{proof}
After applying the change of variables $z=e^{i\theta}$ to the left-hand side of \eqref{lol10}, the integrand has simple poles only at $0$, $\infty$, and whenever
\begin{align*}
\frac{-2\pi c}{(c-a)(i+z)}=2\pi i k \qquad \text{for some $k\in \Z$},
\end{align*}
i.e., if and only if $z=0$, $z=\infty$, or $k\neq 0$ and $z=m_{\alpha_{k}}$, where $m_{\alpha_{k}}$ is given by \eqref{def m alpha k}. Since $|m_{\alpha_{k}}|>1$ for $k<\frac{\frac{c}{2}}{c-a}\notin \N$ and $|m_{\alpha_{k}}|<1$ otherwise, by deforming the contour towards $0$ we get after a residue computation
\begin{align*}
\int_{-\pi}^{\pi} \frac{1}{1-e^{\frac{-2\pi c}{(c-a)(i+e^{i\theta})}}} d\theta = \frac{2\pi}{1-e^{\frac{2\pi i c}{c-a}}} + \frac{ci}{c-a} \sum_{k=\lceil \frac{\frac{c}{2}}{c-a} \rceil}^{+\infty} \frac{1}{k}\frac{1}{k-\frac{c}{c-a}},
\end{align*}
and the claim readily follows. 
\end{proof}

\begin{lemma}\label{lemma:exp2}
Assume that $\frac{c}{c-a}\notin \N$. Then
\begin{align}\label{lol11}
\int_{-\pi}^{\pi} \frac{1}{1-e^{\frac{2\pi i}{c-a}\frac{ci-ai-ae^{i\theta}}{i+e^{i\theta}}}} d\theta = \pi - i \pi \cot \Big( \frac{c\pi}{c-a} \Big) - \frac{ci}{c-a} \sum_{k=\lceil \frac{\frac{c}{2}}{c-a} \rceil}^{+\infty} \frac{1}{k}\frac{1}{k-\frac{c}{c-a}}
\end{align}
\end{lemma}
\begin{proof}
A calculation shows that the integrand in (\ref{lol11}) is the complex conjugate of the integrand in (\ref{lol10}). Hence, the desired conclusion follows from Lemma \ref{lol10}.
\end{proof}

\begin{lemma}\label{lemma:exp3}
Assume that $\frac{c}{c-a}\notin \N$. Then
\begin{align}\label{lol12}
\int_{-\pi}^{\pi} \frac{1}{2-e^{\frac{-2\pi c}{(c-a)(i+e^{i\theta})}}-e^{\frac{2\pi c}{(c-a)(i+e^{i\theta})}}} d\theta = \frac{\pi}{1-\cos \frac{2\pi c}{c-a}} - \frac{c}{\pi (c-a)} \sum_{k=\lceil \frac{\frac{c}{2}}{c-a} \rceil}^{+\infty} \frac{1}{k^{2}}\frac{k-\frac{\frac{c}{2}}{c-a}}{(k-\frac{c}{c-a})^{2}}
\end{align}
\end{lemma}
\begin{proof}
We first apply the change of variables $z=e^{i\theta}$ to the left-hand side of \eqref{lol10}. Then the integrand has simple poles at $0$ and $\infty$, and double poles at $z=m_{\alpha_{k}}$, $k\in \Z\setminus\{0\}$, where $m_{\alpha_{k}}$ is given by \eqref{def m alpha k}. By deforming the contour towards $0$ and computing the residues at $z=m_{\alpha_{k}}$ for $k>\frac{\frac{c}{2}}{c-a}$, we find \eqref{lol12}.
\end{proof}

Combining Lemmas \ref{lemma:pkc}, \ref{lemma:pka}, \ref{lemma:exp1}, \ref{lemma:exp2}, \ref{lemma:exp3} with \eqref{lol13}, after simplifications we find for
$\frac{c}{c-a}\notin \N$ that
\begin{align}
& \int_{-\frac{\pi}{2}}^{\frac{3\pi}{2}} (1+\sin \theta) d\nu(ai+ae^{i\theta}) = \frac{c}{c-a} \bigg\{ \sum_{k=1}^{+\infty} \bigg[ \frac{a^{2}c}{c-a} \frac{1}{(k-\frac{c}{c-a})^{2}} + \frac{ac^{2}}{c-a} \frac{1}{(k+\frac{a}{c-a})^{2}} \nonumber \\
& - \frac{a^{2}+c^{2}}{k-\frac{c}{c-a}} + \frac{c^{2}}{k+\frac{a}{c-a}} + \frac{a^{2}}{k+\frac{c}{c-a}} \bigg] - \pi (c^{2}+a^{2}) \cot (\tfrac{c\pi}{c-a}) -  \frac{a^{2}c}{c-a} \frac{2\pi^{2}}{1-\cos \frac{2\pi c}{c-a}}\bigg\}. \label{lol14}
\end{align}
Using the relations \cite[eqs 1.421.3 and 1.422.4]{GRtable}, namely
\begin{align*}
& \cot(\pi x) = \frac{1}{\pi x} + \frac{2x}{\pi}\sum_{k=1}^{+\infty} \frac{1}{x^{2}-k^{2}}, & & \frac{2}{1-\cos (2\pi x)} = \frac{1}{\sin^{2}(\pi x)} = \frac{1}{\pi^{2}x^{2}}+\frac{2}{\pi^{2}}\sum_{k=1}^{+\infty} \frac{x^{2}+k^{2}}{(x^{2}-k^{2})^{2}},
\end{align*}
which are both valid for $x\in \R\setminus \Z$, we can write \eqref{lol14} as
\begin{align}\nonumber
\int_{-\frac{\pi}{2}}^{\frac{3\pi}{2}} (1+\sin \theta) d\nu(ai+ae^{i\theta}) = &\; \frac{ac^{3}}{(c-a)^{2}} \sum_{k=1}^{+\infty} \frac{1}{(k+\frac{a}{c-a})^{2}} - \frac{a^{2}c^{2}}{(c-a)^{2}} \sum_{k=1}^{+\infty} \frac{1}{(k+\frac{c}{c-a})^{2}} 
	\\ \label{lol15}
& + \frac{c^{3}}{c-a}\sum_{k=1}^{+\infty} \bigg( \frac{1}{k+\frac{a}{c-a}} - \frac{1}{k+\frac{c}{c-a}} \bigg) -2a^{2}-c^{2}
\end{align}
for $\frac{c}{c-a}\notin \N$. Let $a>0$ be fixed, let $c_{\star}>a$ be such that $\frac{c_{\star}}{c_{\star}-a}\in \N$. The right-hand side of \eqref{lol15} is continuous at $c=c_{\star}$. A direct analysis using \eqref{nu on Ca in proof} shows that the left-hand side of (\ref{lol15}) also is continuous at $c = c_{\star}$. We conclude that (\ref{lol15}) is valid for all $c>a>0$.

The analysis for $\int_{-\frac{\pi}{2}}^{\frac{3\pi}{2}} (1+\sin \theta) d\nu(ci+ce^{i\theta})$ is similar (so we omit details), and we find
\begin{multline}\label{lol16}
\int_{-\frac{\pi}{2}}^{\frac{3\pi}{2}} (1+\sin \theta) d\nu(ci+ce^{i\theta}) 
= \frac{-a^{2}c^{2}}{(c-a)^{2}} \sum_{k=1}^{+\infty} \frac{1}{(k+\frac{a}{c-a})^{2}} + \frac{a^{3}c}{(c-a)^{2}} \sum_{k=1}^{+\infty} \frac{1}{(k+\frac{c}{c-a})^{2}} \\
+ \frac{a^{3}}{c-a}\sum_{k=1}^{+\infty} \bigg( \frac{1}{k+\frac{a}{c-a}} - \frac{1}{k+\frac{c}{c-a}} \bigg) + c^{2}.
\end{multline}
Substituting \eqref{lol15} and \eqref{lol16} into (\ref{intOmegaacQnu}), we find that the first integral in (\ref{C in computation}) is given by
\begin{align}\nonumber
 \int_{\partial \Omega_{a,c}}Q(z)d\nu(z)
 =& - \frac{2a^2 c^3}{c-a} \sum_{k=1}^{+\infty} \frac{1}{(k+\frac{a}{c-a})^{2}}
 + \frac{2a^3 c^2}{c-a} \sum_{k=1}^{+\infty} \frac{1}{(k+\frac{c}{c-a})^{2}} 
 	\\ \label{firstintegral1}
& + \frac{2a^2 c^2(c+a)}{c-a} \sum_{k=1}^{+\infty} \bigg( \frac{1}{k+\frac{a}{c-a}} - \frac{1}{k+\frac{c}{c-a}} \bigg)  + 2 c^4 -4 a^4-2 a^2 c^2.
\end{align}

\subsection{Final steps}
The digamma function and its first derivative can be written as (see \cite[5.2.2, 5.7.6]{NIST})
\begin{align}\label{def of psi and psip1p}
& \psi(z) = -\gamma_{\mathrm{E}} + \sum_{m=0}^{+\infty} \frac{z-1}{(m+1)(m+z)}, \qquad \psi'(z) = \sum_{m=0}^{+\infty} \frac{1}{(m+z)^{2}}, \qquad \mbox{for all } z>0,
\end{align}
where $\gamma_{\mathrm{E}}\approx 0.577$ is Euler's gamma constant.
Using these identities, we can write (\ref{firstintegral1}) as
\begin{align}\nonumber
 \int_{\partial \Omega_{a,c}}Q(z)d\nu(z)
 =& - \frac{2a^2 c^3}{c-a} \psi'\left(\frac{a}{c-a}\right)
 + \frac{2a^3 c^2}{c-a} \psi'\left(\frac{c}{c-a}\right)
 	\\ \nonumber
& + \frac{2a^2 c^2(c+a)}{c-a} \bigg[\psi\left(\frac{c}{c-a}+1\right)-\psi\left(\frac{a}{c-a}+1\right)\bigg]
   	\\ \label{firstintegral}
& -2 \left(a^4+a^3 c+a^2 c^2+a c^3-2 c^4\right).
\end{align}
Substituting (\ref{lol17}) and (\ref{firstintegral}) into (\ref{C in computation}) (evaluated at $\tau = 0$) and then using (\ref{Cscaling}), we arrive at
\begin{align*}
C(\tau,z_{0},\theta_{0})= &\; \frac{1}{(1-\tau^{2})^{2}}\frac{\beta}{4}\bigg\{ - \frac{2a^2 c^3}{c-a} \psi'\left(\frac{a}{c-a}\right)
 + \frac{2a^3 c^2}{c-a} \psi'\left(\frac{c}{c-a}\right)
 	\\ 
& + \frac{2a^2 c^2(c+a)}{c-a} \bigg[\psi\left(\frac{c}{c-a}+1\right)-\psi\left(\frac{a}{c-a}+1\right)\bigg]
   	\\
& -2 \left(a^4+a^3 c+a^2 c^2+a c^3-2 c^4\right)
 -  \frac{3}{2}(c^{4}-a^{4}) \bigg\}.
\end{align*}
Simplification gives
\begin{align*}
C(\tau,z_{0},\theta_{0})= \frac{\beta c^{4}}{8(1-\tau^{2})^{2}}F\bigg(\frac{a}{c}\bigg),
\end{align*}
where
\begin{align}\nonumber
F(x) =&\; 5-4x - 4 x^2 - 4x^3  - x^4 -\frac{4x^2}{1-x}\psi'\bigg(\frac{1}{1-x} - 1\bigg) + \frac{4x^3}{1-x}\psi'\bigg(\frac{1}{1-x}\bigg)
	\\\label{Foriginal}
& + \frac{4x^{2}(1+x)}{1-x}\bigg\{ \psi\bigg(\frac{1}{1-x} + 1\bigg) - \psi\bigg(\frac{1}{1-x}\bigg) \bigg\}. 
\end{align}
Using the identities
$$\begin{cases}
\psi(y+ 1) = \frac{1}{y} + \psi(y), \\
 \psi'(y+ 1) = -\frac{1}{y^2} + \psi'(y), 
 \end{cases} \qquad \text{for $y > 0$},$$
which follow from repeated differentiation of $\log \Gamma(y+1) = \log{y} + \log \Gamma(y)$, $y > 0$, we see that the function $F(x)$ in (\ref{Foriginal}) can be written as in (\ref{def of F}). 
This finishes the proof of Theorem \ref{thm:EG tacnode C}.

\appendix

\section{Estimates of harmonic measure}\label{estimateapp}
Suppose $\Omega$ is an open subset of $\C^*$ fulfilling $(i)$ and $(ii)$ of Theorem \ref{mainth2} with $z_0 = 0$. After shrinking $\rho_0$ if necessary, we may assume that the components $\{U_j\}_1^m$ of $\Omega \cap B_{\rho_0}(z_0)$ are such that $U_j \cap \partial B_r(0)$ is a connected curve for every $r \in (0, \rho_0)$ and $j =1, \dots, m$, and such that $\bar{\Omega} \cap \overline{B_{\rho_0}(0)} = \cup_{j=1}^m \bar{U}_j$.

\begin{lemma}\label{extremaldistancelemma}
Fix $j \in \{1, \dots, m\}$ and let $0 < r_0 < R_0 \leq \rho_0$. Let $z_1 \in \Omega$ be such that $|z_1| \geq R_0$. If $r \Theta_j(r)$ is the length of the arc $U_j \cap \partial B_r(0)$, then
\begin{align}\label{omegaleq8pi}
\omega(z_1, \partial U_j \cap B_{r_0}(0), \Omega) \leq \frac{8}{\pi} e^{-\pi \int_{r_0}^{R_0} \frac{dr}{r \Theta_j(r)}}.
\end{align}
\end{lemma}
\begin{proof}
The idea of the proof is to estimate extremal distance (see e.g. \cite[Theorem IV.6.2]{GM2005} and \cite[Theorem H.8]{GM2005}). In the case of an arbitrary finite number of corners but no cusps at $z_0$, a proof is presented in \cite[Lemma 4.2]{CL Corner}. The same proof applies also when cusps are present.
\end{proof}

For $j = 1, \dots, m_1$, let $\pi \alpha_j$ be the opening angle at the corner of $U_j$ at $0$, and let $\gamma_j > 0$ be the corresponding H\"older exponent. For $j = m_1 + 1, \dots, m$, let $d_j> 0$ and $a_{j}>0$ be the order and coefficient of tangency, respectively, of the cusp of $U_j$ at $0$. 

\begin{lemma}\label{Thetajlemma}
Shrinking $\rho_0 >0$ if necessary, there is a constant $C_j > 0$ such that, for $j = 1, \dots, m$ and $0 < r \leq \rho_0$,
\begin{align}\label{Thetajestimate}
\Theta_j(r) \leq \begin{cases}
\pi \alpha_j(1 + C_j r^{\gamma_j}), & j = 1, \dots, m_1,
	\\ 
a_j r^{d_j} (1+C_{j}r^{d_{j}}), & j = m_1 +1, \dots, m.
\end{cases}	
\end{align}
\end{lemma}
\begin{proof}
For $j = 1, \dots, m_1$, the statement follows from \cite[Lemma 3.1]{CL Corner}.
For $j = m_1 +1, \dots, m$, the statement is a consequence of the definition of the order and coefficient of tangency of a cusp. 
\end{proof}

The proofs of our main results use the following lemma.

\begin{lemma}\label{omegaupperboundlemma}
Suppose $\Omega$ is an open subset of $\C^*$ fulfilling $(i)$ and $(ii)$ of Theorem \ref{mainth2} with $z_0 = 0$ and $m \geq 1$. 
Fix $j \in \{1, \dots, m\}$ and let $\alpha_j$, $C_j$, $\gamma_j$, $d_j$ and $a_{j}$ be as in (\ref{Thetajestimate}). If $r > 0$ and $z \in \Omega$ are such that $0 < r < |z| \leq \rho_0$, then
\begin{align}\label{omegaleq8pirzmultiple}
\omega(z, \partial U_j \cap B_{r}(0), \Omega) \leq 
\begin{cases}
\frac{8}{\pi} \big(\frac{r}{|z|}\big)^{\frac{1}{\alpha_j}} 
\big(1 +  C_j |z|^{\gamma_j }\big)^{\frac{1}{\alpha_j \gamma_j}}, & j = 1, \dots, m_1,
	\\
\frac{8}{\pi}  \exp\big( -\pi \frac{r^{-d_j}-|z|^{-d_j}}{a_j d_j} \big) \big( \frac{1+C_{j}r^{d_{j}}}{1+C_{j}|z|^{d_{j}}} \frac{|z|^{d_{j}}}{r^{d_{j}}} \big)^{\frac{\pi C_{j}}{a_{j}d_{j}}},  & j = m_1 +1, \dots, m,
\end{cases}
\end{align}
while if $r > 0$ and $z \in \Omega$ are such that $0 < r < \rho_0 \leq |z|$, then
\begin{align}\label{omegaleq8pirzmultiple2}
\omega(z, \partial U_j \cap B_{r}(0), \Omega) \leq 
\begin{cases}
\frac{8}{\pi} \big(\frac{r}{\rho_0}\big)^{\frac{1}{\alpha_j}} 
\big(1 +  C_j \rho_0^{\gamma_j }\big)^{\frac{1}{\alpha_j  \gamma_j}}, & j = 1, \dots, m_1,
	\\
\frac{8}{\pi}  \exp \big( -\pi \frac{r^{-d_j}-\rho_{0}^{-d_j}}{a_j d_j} \big) \big( \frac{1+C_{j}r^{d_{j}}}{1+C_{j}\rho_{0}^{d_{j}}} \frac{\rho_{0}^{d_{j}}}{r^{d_{j}}} \big)^{\frac{\pi C_{j}}{a_{j}d_{j}}},  & j = m_1 +1, \dots, m.
\end{cases}
\end{align}
\end{lemma}
\begin{proof}
For any $r_0 \in (0, \rho_0)$, Lemma \ref{extremaldistancelemma} yields
\begin{align}\label{omegazEOmega}
\omega(z, \partial U_j \cap B_{r_0}(0), \Omega) \leq \frac{8}{\pi} e^{-\pi \int_{r_0}^{R_0} \frac{dr}{r \Theta_j(r)}} \quad \text{for all $z \in \Omega$ with $|z| > r_0$},
\end{align}
where $R_0 := \min(\rho_0, |z|)$.
Suppose first that $j \in \{m_1 +1, \dots, m\}$.
Using (\ref{Thetajestimate}), we see that
\begin{align*}
\int_{r_0}^{R_{0}} \frac{dr}{r \Theta_j(r)}
& \geq \int_{r_0}^{R_0} \frac{dr}{a_j r^{d_j+1}(1+C_{j}r^{d_{j}})}
= \frac{r_0^{-d_j} - R_{0}^{-d_{j}}}{a_j d_j}+ \frac{C_{j}}{a_{j}d_{j}} \log \bigg( \frac{1+C_{j}R_{0}^{d_{j}}}{1+C_{j}r_{0}^{d_{j}}} \frac{r_{0}^{d_{j}}}{R_{0}^{d_{j}}} \bigg).
\end{align*}
Employing this inequality in (\ref{omegazEOmega}) and replacing $r_0$ by $r$ and $R_{0}$ by $\min(\rho_0, |z|)$, we find the claim. A similar argument using (\ref{Thetajestimate}) gives the desired conclusion also in the case when $j \in \{1, \dots, m_1\}$ (see also \cite[Lemma 4.3]{CL Corner}). 
\end{proof}

\begin{remark}\label{epsilonremark}
Lemma \ref{omegaupperboundlemma} gives slightly stronger estimates at a cusp than what we need in this paper.
Repeating the proof of Lemma \ref{omegaupperboundlemma} with $C_j = 0$ and $a_j$ replaced by $(1+\epsilon)a_j$, we see that if $j = m_1 +1, \dots, m$ and $\Theta_j(r) \leq (1+\epsilon) a_j r^{d_j}$ for all sufficiently small $r > 0$, then 
\begin{align}\label{omegaestimate}
\omega(z, \partial U_j \cap B_{r}(0), \Omega) \leq 
\begin{cases}
\frac{8}{\pi} \exp\big( -\pi \frac{r^{-d_j}-|z|^{-d_j}}{(1+\epsilon)a_j d_j} \big)   & \text{if $0 < r < |z| \leq \rho_0$},
	\\
\frac{8}{\pi} \exp \big( -\pi \frac{r^{-d_j}-\rho_{0}^{-d_j}}{(1+\epsilon)a_j d_j} \big) & \text{if $0 < r < \rho_0 \leq |z|$}.
\end{cases}
\end{align}
It is this weaker form of the estimate that we use in the proof of the upper bound in Section \ref{proofsec}.
\end{remark}

\section{Proof of Proposition \ref{Fprop}}\label{Fapp}
The expansions in (\ref{Fnear0and1}) follow from \cite[5.15.2--5.15.7]{NIST} and the asymptotic formula \cite[5.15.8]{NIST}. It follows that $F(x)$ extends to a continuous function of $x \in [0,1]$ with $F(0) = 1$ and $F(1) = 0$. 
It only remains to show that $F'(x) < 0$ for $x \in (0,1)$. 
To show this, note that
$$F'\bigg(\frac{y-1}{y}\bigg) = -\frac{4 (y-1)}{y^3}G(y) \qquad \text{for $y \in (1, +\infty)$},$$
where
$$G(y) := (y-1) y^3 \psi''(y)+ (y-1)^2 + 2 y^2 \psi'(y).$$
Thus it is enough to show that $G(y) > 0$ for $y \in (1, +\infty)$. Suppose we can show that
\begin{align}\label{psi1estimate}
& \psi'(y) \geq \frac{1}{y} && \text{for $y \in (1, +\infty)$},
	\\ \label{psi2estimate}
& \psi''(y) \geq -\frac{1}{(y-1/2)^2}&& \text{for $y \in (1, +\infty)$}.
\end{align}
Then 
$$G(y) \geq -(y-1) y^3 \frac{1}{(y-1/2)^2}+ (y-1)^2 + 2 y = \frac{5 y^2-4 y+1}{(1-2 y)^2} > 0 \qquad \text{for $y \in (1, +\infty)$},$$
which gives the desired conclusion. So it only remains to show (\ref{psi1estimate}) and (\ref{psi2estimate}).
As $y \to +\infty$, we have
$$\psi'(y) - \frac{1}{y} = \frac{1}{2y^2} + \bigO(y^{-3}), \qquad
\psi''(y) + \frac{1}{(y-\frac{1}{2})^2}= \frac{1}{4y^4} + \bigO(y^{-5}),$$
implying that (\ref{psi1estimate}) and (\ref{psi2estimate}) hold for all sufficiently large positive $y$. 
Furthermore, the relations
$$\psi'(y+1) = -\frac{1}{y^2} + \psi'(y), \qquad
\psi''(y+1) = \frac{2}{y^3} + \psi''(y),$$
imply that, for all $y > 0$,
\begin{align}\label{psi1relation}
\bigg(\psi'(y+1) - \frac{1}{y+1}\bigg) - \bigg(\psi'(y) - \frac{1}{y}\bigg) 
= -\frac{1}{y^2} - \frac{1}{y+1} + \frac{1}{y}  = -\frac{1}{y^2 + y^3} < 0
\end{align}
and, for all $y > \frac{1}{2\sqrt{2}}$,
\begin{align}\nonumber
\bigg(\psi''(y+1) + \frac{1}{(y+1-\frac{1}{2})^2}\bigg) - \bigg(\psi''(y) + \frac{1}{(y-\frac{1}{2})^2}\bigg) 
& = \frac{2}{y^3} + \frac{1}{(y+1-\frac{1}{2})^2} - \frac{1}{(y-\frac{1}{2})^2} 
	\\\label{psi2relation}
& = \frac{2-16 y^2}{y^3 \left(1-4 y^2\right)^2} < 0.
\end{align}
The inequality (\ref{psi1relation}) shows that if (\ref{psi1estimate}) holds at $y +1 > 1$, then it holds also at $y$. It follows that (\ref{psi1estimate}) holds for all $y > 0$.
Similarly, (\ref{psi2relation}) shows that if (\ref{psi2estimate}) holds at $y + 1 > 1 + \frac{1}{2\sqrt{2}}$, then it holds also at $y$, implying that (\ref{psi2estimate}) holds for all $y > \frac{1}{2\sqrt{2}}$.
This completes the proof of (\ref{psi1estimate}) and (\ref{psi2estimate}), and hence also of Proposition \ref{Fprop}.

\subsection*{Acknowledgements}
CC acknowledges support from the Swedish Research Council, Grant No. 2021-04626, and JL acknowledges support from the Swedish Research Council, Grant No. 2021-03877.


\footnotesize
\begin{thebibliography}{99}
\bibitem{A2018} K. Adhikari, Hole probabilities for $\beta$-ensembles and determinantal point processes in the complex plane, \textit{Electron. J. Probab.} \textbf{23} (2018), Paper No. 48, 21 pp.

\bibitem{AR2017} K. Adhikari and N.K. Reddy, Hole probabilities for finite and infinite Ginibre ensembles, \textit{Int. Math. Res. Not. IMRN} (2017), no.21, 6694--6730.

\bibitem{APS2009} G. Akemann, M.J. Phillips, and L. Shifrin, Gap probabilities in non-Hermitian random matrix theory, \textit{J. Math. Phys.} \textbf{50} (2009), no. 6, 063504, 32 pp.

\bibitem{AS2019}
H. Aleksanyan and H. Shahgholian, Discrete Balayage and boundary sandpile, {\it J. Anal. Math.} {\bf 138} (2019), 361--403.

\bibitem{ACC2023} Y. Ameur, C. Charlier, and J. Cronvall, Random normal matrices: eigenvalue correlations near a hard wall, arXiv:2306.14166.

\bibitem{BBMP2009}
A. Bj\"orn, J. Bj\"orn, T. M\"ak\"al\"ainen, and M. Parviainen, Nonlinear balayage on metric spaces, {\it Nonlinear Anal.} {\bf 71} (2009), 2153--2171.

\bibitem{BFreview} S.-S. Byun and P.J. Forrester, Progress on the study of the Ginibre ensembles I: GinUE, arXiv:2211.16223.

\bibitem{BP2024} S.-S. Byun and S. Park, Large gap probabilities of complex and symplectic spherical ensembles with point charges, arXiv:2405.00386.

\bibitem{C2021} C. Charlier, Large gap asymptotics on annuli in the random normal matrix model, \textit{Math. Ann.} \textbf{388} (2024), no. 4, 3529--3587.

\bibitem{C2023} C. Charlier, Hole probabilities and balayage of measures for planar Coulomb gases, arXiv:2311.15285.

\bibitem{CL Corner} C. Charlier and J. Lenells, Balayage of measures: behavior near a corner, arXiv:2403.02964.

\bibitem{D2001}
J.L. Doob, Classical potential theory and its probabilistic counterpart, Springer-Verlag, Berlin, 2001.

\bibitem{ForresterHoleProba} P.J. Forrester, Some statistical properties of the eigenvalues of complex random matrices, \textit{Phys. Lett. A} \textbf{169} (1992), no. 1-2, 21--24. 

\bibitem{GM2005} J.B. Garnett and D.E. Marshall, {\it Harmonic measure}, Cambridge University Press, 2005.

\bibitem{GRtable} 
I.S. Gradshteyn and I.M. Ryzhik, Table of integrals, series, and products. Seventh edition. Elsevier Academic Press, Amsterdam, 2007.

\bibitem{GK2021} A. Groot and A.B.J. Kuijlaars, Matrix-valued orthogonal polynomials related to hexagon tilings, \textit{J. Approx. Theory} \textbf{270} (2021), Paper No. 105619, 36 pp.

\bibitem{G1997}
B. Gustafsson, Direct and inverse balayage---some new developments in classical potential theory, {\it Nonlinear Anal.} {\bf 30} (1997), 2557--2565.

\bibitem{G2004}
B. Gustafsson, Lectures on balayage, {\it Univ. Joensuu Dept. Math. Rep. Ser.} {\bf 7}, University of Joensuu, Joensuu, (2004) 17--63.

\bibitem{GR2018}
B. Gustafsson and J. Roos, Partial balayage on Riemannian manifolds, {\it J. Math. Pures Appl.} {\bf 118} (2018), 82--127.

\bibitem{GS1994}
B. Gustafsson and M. Sakai, Properties of some balayage operators, with applications to quadrature domains and moving boundary problems, {\it Nonlinear Anal.} {\bf 22} (1994), 1221--1245.

\bibitem{JLM1993} B. Jancovici, J. Lebowitz, and G. Manificat, Large charge fluctuations in classical Coulomb systems, \textit{J. Statist. Phys.} \textbf{72} (1993), no. 3-4, 773--787.

\bibitem{K2010}
T. Kaiser, Asymptotic behaviour of the mapping function at an analytic cusp with small perturbation of angles, {\it Comput. Methods Funct. Theory} {\bf 10} (2010), 35--47.

\bibitem{LMS2018} B. Lacroix-A-Chez-Toine, S.N. Majumdar and Gr\'{e}gory Schehr, Rotating trapped fermions in two dimensions and the complex Ginibre ensemble: Exact results for the entanglement entropy and number variance, \textit{Phys. Rev. A} \textbf{99} (2019), 021602.

\bibitem{Land1972} N.S. Landkof, Foundations of modern potential theory, Springer-Verlag, New York-Heidelberg, 1972.

\bibitem{Nehari} Z. Nehari, \textit{Conformal mapping}, McGraw-Hill Book Co., Inc., New York-Toronto-London, 1952.

\bibitem{NW2023} A. Nishry and A. Wennman, The forbidden region for random zeros: Appearance of quadrature domains, \textit{Comm. Pure Appl. Math.} \textbf{77} (2024), no. 3, 1766--1849.

\bibitem{NIST} F.W.J. Olver, A.B. Olde Daalhuis, D.W. Lozier, B.I. Schneider, R.F. Boisvert, C.W. Clark, B.R. Miller, and B.V. Saunders, NIST Digital Library of Mathematical Functions. http://dlmf.nist.gov/, Release 1.2.1 of 2024-06-15.

\bibitem{P1992} C. Pommerenke, {\it Boundary behaviour of conformal maps}, Springer-Verlag, 1992.

\bibitem{P2017}
D. Prokhorov, Conformal mapping asymptotics at a cusp, {\it Rev. Mat. Complut.} {\bf 30} (2017), 79--89.

\bibitem{SaTo} E.B. Saff and V. Totik, {\em Logarithmic Potentials with External Fields},  Grundlehren der Mathematischen Wissenschaften, Springer-Verlag, Berlin, 1997.

\bibitem{Z2022}
N. Zorii, Balayage of measures on a locally compact space, {\it Anal. Math.} {\bf 48} (2022), 249--277.

\bibitem{Z2023}
N. Zorii, On the theory of balayage on locally compact spaces, {\it Potential Anal.} {\bf 59} (2023), 1727--1744.

\bibitem{Z2023b}
N. Zorii, On the theory of capacities on locally compact spaces and its interaction with the theory of balayage, {\it Potential Anal.} {\bf 59} (2023), 1345--1379.

\end{thebibliography}
\end{document}